\documentclass[reqno]{amsart}
\usepackage{amsfonts}
\usepackage[reqno]{amsmath}
\usepackage{amscd}
\usepackage{amssymb}
\usepackage{graphicx,color}
\providecommand{\U}[1]{\protect \rule{.1in}{.1in}}

 \newcommand{\hyper}[5]{\,{}_{#1}F_{#2}\left( 
\begin{array}{r|l}
\begin{array}{cc}{\displaystyle{#3}}\\
{\displaystyle{#4}}
\end{array} & {\displaystyle{#5}}
\end{array} \right)}

\newtheorem{theorem}{Theorem}
\newtheorem{corollary}[theorem]{Corollary}

\theoremstyle{definition}

\theoremstyle{remark}
\newtheorem{remark}{Remark}

\begin{document}
\title[A new family of orthogonal polynomials in three variables]{A new family of orthogonal polynomials in three variables}
\author{Rab\.{I}a Akta\c{S}}
\address{Ankara University, Faculty of Science, Department of Mathematics, 06100, Tando\u{g}an, Ankara, Turkey}
\email{raktas@science.ankara.edu.tr}
\author{Iv\'{a}n Area}
\address{Departamento de Matem\'{a}tica Aplicada II, E.E. Aeron\'{a}utica e do Espazo, Universidade de Vigo, Ourense, Spain}
\email{area@uvigo.gal}
\author{Esra G\"{u}ldo\u{g}an}
\address{Ankara University, Faculty of Science, Department of Mathematics, 06100, Tando\u{g}an, Ankara, Turkey}
\email{eguldogan@ankara.edu.tr}

\begin{abstract}
In this paper we introduce a six-parameter generalization of the four-parameter three-variable polynomials on the simplex and we investigate the properties of these polynomials. Sparse recurrence relations are derived by using ladder relations for shifted univariate Jacobi polynomials and bivariate polynomials on the triangle. Via these sparse recurrence relations, second order partial differential equations are presented. Furthermore, some connection relations are obtained between these polynomials. Finally, new results for the four-parameter three-variable polynomials on the simplex are given.
\end{abstract}
\maketitle

\section{Introduction}\label{sec:1}

The classical univariate Jacobi polynomials $P_{n}^{(a,b)}(x)$ defined by the Rodrigues formula
\[
P_{n}^{(a,b)}\left(x\right)  =\frac{\left(-1\right)  ^{n}} {2^{n}n!}
\left(1-x\right)^{-a} \left(  1+x\right)^{-b}\frac{d^{n}}{dx^{n}}\left \{  (1-x)  ^{n+a}\left(  1+x\right)  ^{n+b}\right \}, \qquad n \geq 0,
\]
are orthogonal with respect to the weight function $w_{a,b}(x)=(1-x)^{a}(1+x)^{b}$ on the interval $(-1,1)$ for $a,b>-1$. The Jacobi polynomials on the interval $(0,1)$, also referred as shifted univariate Jacobi polynomials \cite{Luke}, which we denote by $\widetilde{P}_{n}^{(a,b)}(x):=P_{n}^{\left(  a,b\right)  }\left(  2x-1\right)$,
\begin{equation}\label{eq:sjp}
\widetilde{P}_{n}^{\left(  a,b\right)  }(x)  =\frac{\left(-1\right)  ^{n}}{n!}(1-x)  ^{-a}x^{-b}\frac{d^{n}}{dx^{n}}\left \{  (1-x)  ^{n+a}x^{n+b}\right \}, \qquad n \geq 0,
\end{equation}
are orthogonal on the interval $\left(  0,1\right)  $ with respect to the weight function $\widetilde{w}_{a,b}(x)=(1-x)^{a}x^{b}$ where $a,b>-1$. Indeed, we have
\[
\int_{0}^{1}\widetilde{P}_{n}^{(a,b)}(x)\widetilde{P}_{m}^{(a,b)}(x)(1-x)^{a}x^{b}dx=\delta_{m,n}h_{n}^{\left(  a,b \right)  },
\]
where
\begin{equation} \label{1}
h_{n}^{\left(  a,b \right)  }=\frac{\Gamma \left(  a+n+1\right)\Gamma \left(  b+n+1\right)  }{n!\left(  a+b+2n+1\right)\Gamma \left(  a+b+n+1\right)  },
\end{equation}
with $m,n\in \mathbb{N}_{0}:=\mathbb{N\cup}\left \{  0\right \}$, and $\delta_{m,n}$ is the Kronecker delta \cite{R,S}.

In 1975, Koornwinder \cite{K} gave a general method to derive multivariate orthogonal polynomials from univariate orthogonal polynomials. One of the examples of Koornwinder's method is the Jacobi polynomials on the triangle
\begin{multline}\label{A}
P_{n,k}^{\left(  a,b,c\right)  }\left(  x,y\right)  =P_{n-k}^{\left(2k+b+c+1,a\right)  }\left(  2x-1\right)  (1-x)  ^{k}P_{k}^{\left(  c,b\right)  }\left(  \frac{2y}{1-x}-1\right)  \\
 =\widetilde{P}_{n-k}^{\left(  2k+b+c+1,a\right)  }(x)  \left(1-x\right)  ^{k}\widetilde{P}_{k}^{\left(  c,b\right)  }\left(  \frac{y}{1-x}\right), \quad k=0,1,\dots,n, \quad n=0,1, \dots
\end{multline}
which are orthogonal on the two-dimensional simplex
\begin{equation}\label{eq:T2}
{\mathbb{T}}^{2}:=\{(x,y)\in \mathbb{R}^{2}:x,y>0,1-x-y>0\},
\end{equation}
with respect to the weight function $w_{a,b,c} (x,y)=x^{a}y^{b}(1-x-y)^{c}$ for $a,b,c>-1$,
\begin{multline}
\iint_{{\mathbb{T}}^{2}} P_{n,k}^{\left(  a,b,c\right)  }\left(  x,y\right) P_{r,s}^{\left(  a,b,c\right)  }\left(  x,y\right) x^{a}y^{b}(1-x-y)^{c}dxdy \\
 =\frac{\Gamma \left(  a+n-k+1\right)  \Gamma \left(  b+c+n+k+2\right) \Gamma \left(  b+k+1\right)  \Gamma \left(
c+k+1\right)  }{\left(  n-k\right)  !k!\left(  2n+a+b
+c+2\right) \left(  2k+b+c+1\right)  \Gamma \left(
n+k+a+b+c+2\right)  \Gamma \left(  k+b+c+1\right)}
 \delta_{n,r}\delta_{k,s}.
\end{multline}
In the recent paper \cite{OTV}, the authors introduce a bivariate four-parameter variant of the Koornwinder polynomials on the triangle defined by
\begin{equation}\label{B}
P_{n,k}^{(a,b,c,d)}\left(  x,y\right)  =\widetilde{P}_{n-k}^{\left(  2k+b+c+d+1,a\right)  }(x)  (1-x)^{k}\widetilde{P}_{k}^{\left(  c,b\right)  }\left(  \frac{y}{1-x}\right),
\end{equation}
where $a$, $b$, $c$, $d>-1$, $n$ and $k$ are integers satisfying $0\leq k\leq n.$ These polynomials are orthogonal on the two-dimensional simplex ${\mathbb{T}}^{2}$ defined in \eqref{eq:T2} with respect to the weight function
\[
w_{a,b,c,d}(x,y)=x^{a}y^{b}(1-x-y)^{c}(1-x)  ^{d}.
\]
Indeed, it follows
\begin{multline}
\iint_{{\mathbb{T}}^{2}} P_{n,k}^{(a,b,c,d)}\left(  x,y\right) P_{r,s}^{(a,b,c,d)}\left(  x,y\right) x^{a}y^{b}(1-x-y)^{c}(1-x)  ^{d}dxdy=\frac{\Gamma \left(  a+n-k+1\right) }{\left(  n-k\right)  !k!} \\ \times \frac{ \Gamma \left(  b+c+d+n+k+2\right) \Gamma \left(  b+k+1\right)  \Gamma \left(c+k+1\right)  }{\left(  2n+a+b+c+d+2\right) \left(  2k+b+c+1\right)  \Gamma \left(n+k+a+b+c+d+2\right)  \Gamma \left(  k+b+c+1\right)}  \delta_{n,r}\delta_{k,s}.
\end{multline}

The special case $d=0$ gives the bivariate Koornwinder polynomials on the triangle defined in \eqref{A}. The authors obtain sparse recurrence relations by using some differential relations mapping $P_{n,k}^{(a,b,c,d)}\left(  x,y\right)$ to scalar multiple of $P_{\widetilde{n},\widetilde{k}}^{\left(  \widetilde{a},\widetilde{b},\widetilde{c},\widetilde{d}\right)  }\left(  x,y\right)$, where the new parameters in $P_{\widetilde{n},\widetilde{k}}^{\left(  \widetilde{a},\widetilde{b},\widetilde{c},\widetilde{d}\right)  }\left(  x,y\right)$ are $n$, $k$, $a$, $b$, $c$, or $d$, respectively, or these values incremented or decremented by $1$. This approach allows to derive new recurrence relations for the bivariate (classical) Koornwinder polynomials $P_{n,k}^{\left(a,b,c\right)  }\left(  x,y\right)$ (see \cite{OTV}).

Let
\begin{equation}\label{eq:3dsimplex}
{\mathbb{T}}^{3}:=\{(x,y,z)\in \mathbb{R}^{3} :x,y,z>0,1-x-y-z>0\}
\end{equation}
be the three-dimensional simplex.
In {\cite{DX}} it was introduced the following family of three-variable polynomials three-variable polynomials
\begin{multline}\label{C1}
P_{n}^{\left(  \alpha,\beta,\gamma,\delta\right)  }(x,y,z)  :=P_{n_{1},n_{2},n_{3}}^{\left(  \alpha,\beta,\gamma,\delta\right)
}(x,y,z)    =P_{n_{1}}^{\left(  \beta+\gamma+\delta+2n_{2}+2n_{3}+2,\alpha \right)}\left(  2x-1\right)  (1-x)  ^{n_{2}} \\
  \times P_{n_{2}}^{\left(  \gamma+\delta+2n_{3}+1,\beta \right)  }\left(\frac{2y}{1-x}-1\right)  \left(  1-x-y\right)  ^{n_{3}}P_{n_{3}}^{\left(
\delta,\gamma \right)  }\left(  \frac{2z}{1-x-y}-1\right),
\end{multline}
where $n_{1}$, $n_{2}$, $n_{3}$, $n\in \mathbb{N}_{0}$ and $n_{1}+n_{2}+n_{3}=n$. These polynomials are orthogonal on the simplex ${\mathbb{T}}^{3}$  with respect to the weight function
\[
w_{\alpha,\beta,\gamma,\delta}(x,y,z)=x^{\alpha}y^{\beta}z^{\gamma}(1-x-y-z)^{\delta}.
\]

In this paper, we consider a six-parameter variant of the three-variable polynomials \eqref{C1} defined as
\begin{multline}\label{C}
P_{n}^{(\alpha,\beta,\gamma,\delta,a,b)  }(x,y,z)  :=P_{n_{1},n_{2},n_{3}}^{(\alpha,\beta,\gamma,\delta,a,b)
}(x,y,z)    =P_{n_{1}}^{\left(  \beta+\gamma+\delta+a+b+2n_{2}+2n_{3}+2,\alpha \right)}\left(  2x-1\right)  (1-x)  ^{n_{2}} \\
  \times P_{n_{2}}^{\left(  \gamma+\delta+2n_{3}+b+1,\beta \right)  }\left(\frac{2y}{1-x}-1\right)  \left(  1-x-y\right)  ^{n_{3}}P_{n_{3}}^{\left(
\delta,\gamma \right)  }\left(  \frac{2z}{1-x-y}-1\right),
\end{multline}
where $n_{1}$, $n_{2}$, $n_{3}$, $n\in \mathbb{N}_{0}$ and $n_{1}+n_{2}+n_{3}=n$. The polynomials in \eqref{C} are orthogonal with respect to weight function
\[
w_{\alpha,\beta,\gamma,\delta,a,b}(x,y,z)=x^{\alpha}y^{\beta}z^{\gamma}(1-x-y-z)^{\delta}(1-x)  ^{a}\left(  1-x-y\right)  ^{b}
\]
on the simplex ${\mathbb{T}}^{3}$ defined in \eqref{eq:3dsimplex}. More precisely,
\begin{multline*}
{\displaystyle \iiint \limits_{{\mathbb{T}}^{3}}} w_{\alpha,\beta,\gamma,\delta,a,b}(x,y,z)P_{n_{1},n_{2},n_{3}}^{\left(\alpha,\beta,\gamma,\delta,a,b\right)  }(x,y,z)  P_{m_{1},m_{2},m_{3}}^{(\alpha,\beta,\gamma,\delta,a,b)  }(x,y,z)  dxdydz\\
  =h_{n_{1}}^{\left(  \beta+\gamma+\delta+a+b+2n_{2}+2n_{3}+2,\alpha \right)}h_{n_{2}}^{\left(  \gamma+\delta+2n_{3}+b+1,\beta \right)  }h_{n_{3}}^{\left(
\delta,\gamma \right)  }\delta_{m_{1},n_{1}}\delta_{m_{2},n_{2}}\delta_{m_{3},n_{3}},
\end{multline*}
where $h_{n}^{\left(  a,b \right)  }$ is given in \eqref{1}. The special case of $a=b=0$ gives the orthogonal polynomials on the simplex ${\mathbb{T}}^{3}~$\cite{DX}.

The structure of this paper is as follows. In the next section, we remind some differential relations and the corresponding sparse recurrence relations given in \cite{OTV} for univariate (shifted) Jacobi polynomials $\widetilde{P}_{n}^{\left(  a,b\right)}(x)$ and the bivariate polynomials $P_{n,k}^{\left(  a,b,c,d\right)}\left(  x,y\right)$. Some extra relations are also presented for the bivariate polynomials, which are new to the best of our knowledge. In section \ref{sec:3}, we give the corresponding differential relations for the polynomials \eqref{C}. Finally, in section \ref{sec:4} sparse recurrence relations for the polynomials \eqref{C1} are deduced.

\section{Preliminaries}\label{sec:2}

Let \cite{OTV}
\[
\begin{array}[c]{l}
\displaystyle \widetilde{\mathcal{L}}_{1}u=\dfrac{du}{dx},\\
\displaystyle \widetilde{\mathcal{L}}_{2,a,b,n}u=\left(  a+b+n+1\right)  u+x\frac{du}{dx},\\
\displaystyle \widetilde{\mathcal{L}}_{3,a,b,n}u=\left(  a+b+n+1\right)  u-(1-x)\frac{du}{dx},\\
\displaystyle \widetilde{\mathcal{L}}_{4,a,b,n}u=\left(  xa-(1-x)  \left(b+n+1\right)  \right)  u-x(1-x)  \frac{du}{dx},\\
\displaystyle \widetilde{\mathcal{L}}_{5,a,b,n}u=\left(  x\left(  a+n+1\right)  -\left(1-x\right)  b\right)  u-x(1-x)  \frac{du}{dx},\\
\displaystyle \widetilde{\mathcal{L}}_{6,b}u=bu+x\frac{du}{dx},
\end{array}
\begin{array}[c]{c}
\end{array}
\begin{array}[c]{l}
\displaystyle \widetilde{\mathcal{L}}_{1,a,b}^{+}u=\left(  xa-(1-x)  b\right)u-x(1-x)  \frac{du}{dx},\\
\displaystyle \widetilde{\mathcal{L}}_{2,a,n}^{+}u=\left(  a+(1-x)  n\right)u-x(1-x)  \frac{du}{dx},\\
\displaystyle \widetilde{\mathcal{L}}_{3,b,n}^{+}u=\left(  b+xn\right)  u+x(1-x)\frac{du}{dx},\\
\displaystyle \widetilde{\mathcal{L}}_{4,n}^{+}u=-nu+x\frac{du}{dx},\\
\displaystyle \widetilde{\mathcal{L}}_{5,n}^{+}u=nu+(1-x)  \frac{du}{dx},\\
\displaystyle \widetilde{\mathcal{L}}_{6,a}^{+}u=au-(1-x)  \frac{du}{dx}.
\end{array}
\]
Then, the following sparse recurrence relations for Jacobi polynomials hold true
\[
\begin{array}[c]{l}
\widetilde{\mathcal{L}}_{1}\widetilde{P}_{n}^{\left(  a,b\right)  }\left(x\right)  =\left(  n+a+b+1\right)  \widetilde{P}_{n-1}^{\left(a+1,b+1\right)  }(x), \\
\widetilde{\mathcal{L}}_{2,a,b,n}\widetilde{P}_{n}^{\left(  a,b\right)  }\left(x\right)  =\left(  n+a+b+1\right)  \widetilde{P}_{n}^{\left(  a+1,b\right)}(x), \\
\widetilde{\mathcal{L}}_{3,a,b,n}\widetilde{P}_{n}^{\left(  a,b\right)  }\left(x\right)  =\left(  n+a+b+1\right)  \widetilde{P}_{n}^{\left(  a,b+1\right)}(x), \\
\widetilde{\mathcal{L}}_{4,a,b,n}\widetilde{P}_{n}^{\left(  a,b\right)  }\left(x\right)  =\left(  n+1\right)  \widetilde{P}_{n+1}^{\left(  a-1,b\right)}(x), \\
\widetilde{\mathcal{L}}_{5,a,b,n}\widetilde{P}_{n}^{\left(  a,b\right)  }\left(x\right)  =\left(  n+1\right)  \widetilde{P}_{n+1}^{\left(  a,b-1\right)}(x), \\
\widetilde{\mathcal{L}}_{6,b}\widetilde{P}_{n}^{\left(  a,b\right)  }\left(x\right)  =\left(  n+b\right)  \widetilde{P}_{n}^{\left(  a+1,b-1\right)}(x),
\end{array}
\begin{array}[c]{c}
\end{array}
\begin{array}[c]{l}
\widetilde{\mathcal{L}}_{1,a,b}^{+}\widetilde{P}_{n}^{\left(  a,b\right)  }\left(x\right)  =\left(  n+1\right)  \widetilde{P}_{n+1}^{\left(  a-1,b-1\right)}(x), \\
\widetilde{\mathcal{L}}_{2,a,n}^{+}\widetilde{P}_{n}^{\left(  a,b\right)  }\left(x\right)  =\left(  n+a\right)  \widetilde{P}_{n}^{\left(  a-1,b\right)}(x), \\
\widetilde{\mathcal{L}}_{3,b,n}^{+}\widetilde{P}_{n}^{\left(  a,b\right)  }\left(x\right)  =\left(  n+b\right)  \widetilde{P}_{n}^{\left(  a,b-1\right)}(x), \\
\widetilde{\mathcal{L}}_{4,n}^{+}\widetilde{P}_{n}^{\left(  a,b\right)  }\left(x\right)  =\left(  n+b\right)  \widetilde{P}_{n-1}^{\left(  a+1,b\right)}(x), \\
\widetilde{\mathcal{L}}_{5,n}^{+}\widetilde{P}_{n}^{\left(  a,b\right)  }\left(x\right)  =\left(  n+a\right)  \widetilde{P}_{n-1}^{\left(  a,b+1\right)}(x), \\
\widetilde{\mathcal{L}}_{6,a}^{+}\widetilde{P}_{n}^{\left(  a,b\right)  }\left(x\right)  =\left(  n+a\right)  \widetilde{P}_{n}^{\left(  a-1,b+1\right)}(x).
\end{array}
\]
Moreover, the following second order differential relations yield
\begin{align*}
\widetilde {\mathcal{L}}_{1,a,b}^{+}\widetilde{\mathcal{L}}_{1}\widetilde{P}_{n}^{\left(a-1,b-1\right)  }(x)  &=n (n + a + b - 1) \widetilde{P}_{n}^{\left(  a-1,b-1\right)}(x),  \\
\widetilde{\mathcal{L}}_{1}\widetilde{\mathcal{L}}_{1,a,b}^{+}\widetilde{P}_{n}^{\left(  a,b\right)  }(x) &=(n+1)(a+b+n) \widetilde{P}_{n}^{\left(  a,b\right)  }(x) , \\
\widetilde {\mathcal{L}}_{2,a,n}^{+}\widetilde{\mathcal{L}}_{2,a,b,n}\widetilde{P}_{n}^{\left(a-1,b+1\right)  }(x)  &=(n + a) (n + a + b + 1) \widetilde{P}_{n}^{\left(  a-1,b+1\right)}(x),  \\
\widetilde {\mathcal{L}}_{2,a,b,n}\widetilde{\mathcal{L}}_{2,a,n}^{+}\widetilde{P}_{n}^{\left(a,b+1\right)  }(x)  &=(n + a) (n + a + b + 1) \widetilde{P}_{n}^{\left(  a,b+1\right)}(x), \\
\widetilde {\mathcal{L}}_{3,b,n}^{+}\widetilde{\mathcal{L}}_{3,a,b,n}\widetilde{P}_{n}^{\left(a+1,b-1\right)  }(x)  &=(n + b) (n + a + b + 1) \widetilde{P}_{n}^{\left(  a+1,b-1\right)}(x),  \\
\widetilde {\mathcal{L}}_{3,a,b,n}\widetilde{\mathcal{L}}_{3,b,n}^{+}\widetilde{P}_{n}^{\left(a+1,b\right)  }(x)  &=(n + b) (n + a + b + 1) \widetilde{P}_{n}^{\left(  a+1,b\right)}(x), \\
\widetilde {\mathcal{L}}_{4,n}^{+}\widetilde{\mathcal{L}}_{4,a,b,n}\widetilde{P}_{n-1}^{\left(a,b+1\right)  }(x)  &=n(n +  b + 1) \widetilde{P}_{n-1}^{\left(  a,b+1\right)}(x), \\
\widetilde {\mathcal{L}}_{4,a,b,n}\widetilde{\mathcal{L}}_{4,n}^{+}\widetilde{P}_{n}^{\left(a-1,b+1\right)  }(x)  &=n(n + b + 1) \widetilde{P}_{n}^{\left(  a-1,b+1\right)}(x), \\
\widetilde {\mathcal{L}}_{5,n}^{+}\widetilde{\mathcal{L}}_{5,a,b,n}\widetilde{P}_{n-1}^{\left(a+1,b\right)  }(x)  &=n (n + a + 1) \widetilde{P}_{n-1}^{\left(  a+1,b\right)}(x),  \\
\widetilde {\mathcal{L}}_{5,a,b,n}\widetilde{\mathcal{L}}_{5,n}^{+}\widetilde{P}_{n}^{\left(a+1,b-1\right)  }(x) &=n (n + a + 1) \widetilde{P}_{n}^{\left(  a+1,b-1\right)}(x),
\end{align*}
\begin{align*}
\widetilde {\mathcal{L}}_{6,a}^{+}\widetilde{\mathcal{L}}_{6,b}\widetilde{P}_{n}^{\left(a-1,b\right)  }(x)  &=(n+a)(n+b)\widetilde{P}_{n}^{\left(  a-1,b\right)}(x), \\
\widetilde {\mathcal{L}}_{6,b}\widetilde{\mathcal{L}}_{6,a}^{+}\widetilde{P}_{n}^{\left(a,b-1\right)  }(x)  &=(n+a)(n+b) \widetilde{P}_{n}^{\left(  a,b-1\right)}(x).
\end{align*}
{}From the above relations we also have the following second order differential equations for the univariate shifted Jacobi polynomials
\begin{align*}
\widetilde {\mathcal{L}}_{1,a+1,b+1}^{+}\widetilde{\mathcal{L}}_{1}\widetilde{P}_{n}^{\left(a,b\right)  }(x)  &=n (n + a + b + 1) \widetilde{P}_{n}^{\left(  a,b\right)}(x),  \\
\widetilde{\mathcal{L}}_{1}\widetilde{\mathcal{L}}_{1,a,b}^{+}\widetilde{P}_{n}^{\left(  a,b\right)  }(x) &=(n+1)(a+b+n) \widetilde{P}_{n}^{\left(  a,b\right)  }(x) , \\
\widetilde {\mathcal{L}}_{2,a+1,n}^{+}\widetilde{\mathcal{L}}_{2,a+1,b-1,n}\widetilde{P}_{n}^{\left(a,b\right)  }(x)  &=(n + a +1) (n + a + b + 1) \widetilde{P}_{n}^{\left(  a,b\right)}(x),  \\
\widetilde {\mathcal{L}}_{2,a,b-1,n}\widetilde{\mathcal{L}}_{2,a,n}^{+}\widetilde{P}_{n}^{\left(a,b\right)  }(x)  &=(n + a) (n + a + b ) \widetilde{P}_{n}^{\left(  a,b\right)}(x), \\
\widetilde {\mathcal{L}}_{3,b+1,n}^{+}\widetilde{\mathcal{L}}_{3,a-1,b+1,n}\widetilde{P}_{n}^{\left(a,b\right)  }(x)  &=(n + b+1) (n + a + b+1 ) \widetilde{P}_{n}^{\left(  a,b\right)}(x),  \\
\widetilde {\mathcal{L}}_{3,a-1,b,n}\widetilde{\mathcal{L}}_{3,b,n}^{+}\widetilde{P}_{n}^{\left(a,b\right)  }(x)  &=(n + b) (n + a + b ) \widetilde{P}_{n}^{\left(  a,b\right)}(x),  \\
\widetilde {\mathcal{L}}_{4,n+1}^{+}\widetilde{\mathcal{L}}_{4,a,b-1,n+1}\widetilde{P}_{n}^{\left(a,b\right)  }(x)  &=(n + 1) (n + b+1) \widetilde{P}_{n}^{\left(  a,b\right)}(x),  \\
\widetilde {\mathcal{L}}_{4,a+1,b-1,n}\widetilde{\mathcal{L}}_{4,n}^{+}\widetilde{P}_{n}^{\left(a,b\right)  }(x)  &=n(n + b ) \widetilde{P}_{n}^{\left(  a,b\right)}(x), \\
\widetilde {\mathcal{L}}_{5,n+1}^{+}\widetilde{\mathcal{L}}_{5,a-1,b,n+1}\widetilde{P}_{n}^{\left(a,b\right)  }(x)  &=(n+1) (n + a + 1) \widetilde{P}_{n}^{\left(  a,b\right)}(x), \\
\widetilde {\mathcal{L}}_{5,a-1,b+1,n}\widetilde{\mathcal{L}}_{5,n}^{+}\widetilde{P}_{n}^{\left(a,b\right)  }(x) &=n (n + a ) \widetilde{P}_{n}^{\left(  a,b\right)}(x), \\
\widetilde {\mathcal{L}}_{6,a+1}^{+}\widetilde{\mathcal{L}}_{6,b}\widetilde{P}_{n}^{\left(a,b\right)  }(x)  &=(n+a+1)(n+b)\widetilde{P}_{n}^{\left(  a,b\right)}(x), \\
\widetilde {\mathcal{L}}_{6,b+1}\widetilde{\mathcal{L}}_{6,a}^{+}\widetilde{P}_{n}^{\left(a,b\right)  }(x)  &=(n+a)(n+b+1) \widetilde{P}_{n}^{\left(  a,b\right)}(x).
\end{align*}
The first, second and the last two relations give exactly the second order differential equation for the univariate shifted Jacobi polynomials. Moreover, the third, fourth, seventh, and eighth relations give the second order differential equation for the univariate shifted Jacobi polynomials multiplied by $x$, while the fifth, sixth, ninth, and tenth relations give the second order differential equation for the univariate shifted Jacobi polynomials multiplied by $1-x$.

By using these $12$ relations, in \cite{OTV} $24$ relations for the bivariate orthogonal polynomials $P_{n,k}^{(a,b,c,d)}\left(  x,y\right)$ defined in \eqref{B} are derived. These $24$ relations map $P_{n,k}^{(a,b,c,d)}\left(  x,y\right)$ to a scalar multiple of $P_{\widetilde{n},\widetilde{k}}^{\left(  \widetilde{a},\widetilde{b},\widetilde{c},\widetilde{d}\right)  }\left(  x,y\right)$, where the new parameters in $P_{\widetilde{n},\widetilde{k}}^{\left(\widetilde{a},\widetilde{b},\widetilde{c},\widetilde{d}\right)  }\left(x,y\right)  $ are $n,k,a,b,c$ or $d$, respectively, incremented or decremented by $0$ or $1$. Following the notations in \cite{OTV}, let
\[
\begin{array}[c]{l}
\mathcal{M}_{0,1}u=\dfrac{\partial u}{\partial y},\\
\mathcal{M}_{0,2}u=\left(  k+b+c+1\right)  u+y\dfrac{\partial u}{\partial y},\\
\mathcal{M}_{0,3}u=\left(  k+b+c+1\right)  u-(1-x-y)\dfrac{\partial u}{\partial y},\\
\mathcal{M}_{0,4}u=\left(  yc-(1-x-y)\left(  b+k+1\right)  \right)  u-y(1-x-y)\dfrac{\partial u}{\partial y},\\
\mathcal{M}_{0,5}u=\left(  y\left(  c+k+1\right)  -(1-x-y)b\right)  u-y(1-x-y)\dfrac{\partial u}{\partial y},\\
\mathcal{M}_{0,6}u=bu+y\dfrac{\partial u}{\partial y}, \\
\displaystyle \mathcal{M}_{0,1}^{+}u=\left(  yc-(1-x-y)b\right)  u-y(1-x-y)\dfrac{\partial u}{\partial y},\\
\displaystyle \mathcal{M}_{0,2}^{+}u=\left(  c+k-\frac{yk}{1-x}\right)  u-\frac{y(1-x-y)}{1-x}\dfrac{\partial u}{\partial y},\\
\displaystyle \mathcal{M}_{0,3}^{+}u=\left(  b+\frac{ky}{1-x}\right)  u+\frac{y(1-x-y)}{1-x}\dfrac{\partial u}{\partial y},\\
\displaystyle \mathcal{M}_{0,4}^{+}u=-\frac{k}{1-x}u+\frac{y}{1-x}\dfrac{\partial u}{\partial y},
\end{array}
\]
\[
\begin{array}[c]{l}
\displaystyle \mathcal{M}_{0,5}^{+}u=\frac{k}{1-x}u+\left(  1-\frac{y}{1-x}\right)\dfrac{\partial u}{\partial y},\\
\displaystyle \mathcal{M}_{0,6}^{+}u=cu-(1-x-y)\dfrac{\partial u}{\partial y},
\end{array}
\]
and
\[
\begin{array}[c]{l}
\mathcal{M}_{1,0}u=\dfrac{ku}{1-x}+\dfrac{\partial u}{\partial x}-\frac{y}{1-x}\dfrac{\partial u}{\partial y},\\
\mathcal{M}_{1,0}^{+}u=\left(  x\left(  k+a+b+c+d+1\right)  -a\right)u-x(1-x)  \dfrac{\partial u}{\partial x}+xy\dfrac{\partial u}{\partial y},\\
\mathcal{M}_{2,0}u=\left(  n+k+a+b+c+d+2\right)  u+\dfrac{kx}{1-x}u+x\dfrac{\partial u}{\partial x}-\frac{xy}{1-x}\dfrac{\partial u}{\partial y},\\
\mathcal{M}_{2,0}^{+}u=\left(  n+k+b+c+d+1-xn\right)  u-x(1-x)\dfrac{\partial u}{\partial x}+xy\dfrac{\partial u}{\partial y},\\
\mathcal{M}_{3,0}u=\left(  n+a+b+c+d+2\right)  u-(1-x)\dfrac{\partial u}{\partial x}+y\dfrac{\partial u}{\partial y},\\
\mathcal{M}_{3,0}^{+}u=\left(  a+xn\right)  u+x(1-x)\dfrac{\partial u}{\partial x}-xy\dfrac{\partial u}{\partial y},\\
\mathcal{M}_{4,0}u=\left(  x\left(  n+a+b+c+d+2\right)  -a-n+k-1\right)u-x(1-x)  \dfrac{\partial u}{\partial x}+xy\dfrac{\partial u}{\partial y},\\
\mathcal{M}_{4,0}^{+}u=\dfrac{ku}{1-x}-nu+x\dfrac{\partial u}{\partial x}-\frac{xy}{1-x}\dfrac{\partial u}{\partial y},\\
\mathcal{M}_{5,0}u=x\left(  n+a+b+c+d+2\right)  u-au-x(1-x)\dfrac{\partial u}{\partial x}+xy\dfrac{\partial u}{\partial y},\\
\mathcal{M}_{5,0}^{+}u=nu+(1-x)  \dfrac{\partial u}{\partial x}-y\dfrac{\partial u}{\partial y},\\
\mathcal{M}_{6,0}u=au+\dfrac{kx}{1-x}u+x\dfrac{\partial u}{\partial x}-\frac{xy}{1-x}\dfrac{\partial u}{\partial y},\\
\mathcal{M}_{6,0}^{+}u=\left(  k+b+c+d+1\right)  u-(1-x)\dfrac{\partial u}{\partial x}+y\dfrac{\partial u}{\partial y}.
\end{array}
\]
The sparse recurrence relations are given as follows
\begin{align*}
\mathcal{M}_{0,1}P_{n,k}^{(a,b,c,d)}\left(  x,y\right)   &=\left(  k+b+c+1\right)  P_{n-1,k-1}^{\left(  a,b+1,c+1,d\right)  }\left(x,y\right), \\
\mathcal{M}_{0,1}^{+}P_{n,k}^{(a,b,c,d)}\left(  x,y\right)   &=\left(  k+1\right)  P_{n+1,k+1}^{\left(  a,b-1,c-1,d\right)  }\left(x,y\right), \\
\mathcal{M}_{0,2}P_{n,k}^{(a,b,c,d)}\left(  x,y\right)   &=\left(  k+b+c+1\right)  P_{n,k}^{\left(  a,b,c+1,d-1\right)  }\left(x,y\right), \\
\mathcal{M}_{0,2}^{+}P_{n,k}^{(a,b,c,d)}\left(  x,y\right)   &=\left(  k+c\right)  P_{n,k}^{\left(  a,b,c-1,d+1\right)  }\left(  x,y\right), \\
\mathcal{M}_{0,3}P_{n,k}^{(a,b,c,d)}\left(  x,y\right)   &=\left(  k+b+c+1\right)  P_{n,k}^{\left(  a,b+1,c,d-1\right)  }\left(x,y\right), \\
\mathcal{M}_{0,3}^{+}P_{n,k}^{(a,b,c,d)}\left(  x,y\right)   &=\left(  k+b\right)  P_{n,k}^{\left(  a,b-1,c,d+1\right)  }\left(  x,y\right),\\
\mathcal{M}_{0,4}P_{n,k}^{(a,b,c,d)}\left(  x,y\right)   &=\left(  k+1\right)  P_{n+1,k+1}^{\left(  a,b,c-1,d-1\right)  }\left(x,y\right), \\
\mathcal{M}_{0,4}^{+}P_{n,k}^{(a,b,c,d)}\left(  x,y\right)   &=\left(  k+b\right)  P_{n-1,k-1}^{\left(  a,b,c+1,d+1\right)  }\left(x,y\right), \\
\mathcal{M}_{0,5}P_{n,k}^{(a,b,c,d)}\left(  x,y\right)   &=\left(  k+1\right)  P_{n+1,k+1}^{\left(  a,b-1,c,d-1\right)  }\left(x,y\right), \\
\mathcal{M}_{0,5}^{+}P_{n,k}^{(a,b,c,d)}\left(  x,y\right)   &=\left(  k+c\right)  P_{n-1,k-1}^{\left(  a,b+1,c,d+1\right)  }\left(x,y\right), \\
\mathcal{M}_{0,6}P_{n,k}^{(a,b,c,d)}\left(  x,y\right)   &=\left(  k+b\right)  P_{n,k}^{\left(  a,b-1,c+1,d\right)  }\left(  x,y\right),\\
\mathcal{M}_{0,6}^{+}P_{n,k}^{(a,b,c,d)}\left(  x,y\right)   &=\left(  k+c\right)  P_{n,k}^{\left(  a,b+1,c-1,d\right)  }\left(  x,y\right),
\end{align*}
and
\begin{align*}
\mathcal{M}_{1,0}P_{n,k}^{(a,b,c,d)}\left(  x,y\right)   &=\left(  n+k+a+b+c+d+2\right)  P_{n-1,k}^{\left(  a+1,b,c,d+1\right)  }\left(x,y\right), \\
\mathcal{M}_{1,0}^{+}P_{n,k}^{(a,b,c,d)}\left(  x,y\right)   &=\left(  n-k+1\right)  P_{n+1,k}^{\left(  a-1,b,c,d-1\right)  }\left(x,y\right), \\
\mathcal{M}_{2,0}P_{n,k}^{(a,b,c,d)}\left(  x,y\right)   &=\left(  n+k+a+b+c+d+2\right)  P_{n,k}^{\left(  a,b,c,d+1\right)  }\left(x,y\right), \\
\mathcal{M}_{2,0}^{+}P_{n,k}^{(a,b,c,d)}\left(  x,y\right)   &=\left(  n+k+b+c+d+1\right)  P_{n,k}^{\left(  a,b,c,d-1\right)  }\left(x,y\right),
\end{align*}
\begin{align*}
\mathcal{M}_{3,0}P_{n,k}^{(a,b,c,d)}\left(  x,y\right)   &=\left(  n+k+a+b+c+d+2\right)  P_{n,k}^{\left(  a+1,b,c,d\right)  }\left(x,y\right), \\
\mathcal{M}_{3,0}^{+}P_{n,k}^{(a,b,c,d)}\left(  x,y\right)   &=\left(  n-k+a\right)  P_{n,k}^{\left(  a-1,b,c,d\right)  }\left(  x,y\right),\\
\mathcal{M}_{4,0}P_{n,k}^{(a,b,c,d)}\left(  x,y\right)   &=\left(  n-k+1\right)  P_{n+1,k}^{\left(  a,b,c,d-1\right)  }\left(x,y\right), \\
\mathcal{M}_{4,0}^{+}P_{n,k}^{(a,b,c,d)}\left(  x,y\right)   &=\left(  n-k+a\right)  P_{n-1,k}^{\left(  a,b,c,d+1\right)  }\left(x,y\right), \\
\mathcal{M}_{5,0}P_{n,k}^{(a,b,c,d)}\left(  x,y\right)   &=\left(  n-k+1\right)  P_{n+1,k}^{\left(  a-1,b,c,d\right)  }\left(x,y\right), \\
\mathcal{M}_{5,0}^{+}P_{n,k}^{(a,b,c,d)}\left(  x,y\right)   &=\left(  n+k+b+c+d+1\right)  P_{n-1,k}^{\left(  a+1,b,c,d\right)  }\left(x,y\right), \\
\mathcal{M}_{6,0}P_{n,k}^{(a,b,c,d)}\left(  x,y\right)   &=\left(  n-k+a\right)  P_{n,k}^{\left(  a-1,b,c,d+1\right)  }\left(x,y\right), \\
\mathcal{M}_{6,0}^{+}P_{n,k}^{(a,b,c,d)}\left(  x,y\right)   &=\left(  n+k+b+c+d+1\right)  P_{n,k}^{\left(  a+1,b,c,d-1\right)  }\left(x,y\right).
\end{align*}
Moreover, the following second order differential relations hold
\begin{align*}
\mathcal{M}_{0,1}^{+}\mathcal{M}_{0,1}P_{n,k}^{\left(  a,b-1,c-1,d\right)}\left(  x,y\right)   &  =k\left(  k+b+c-1\right)  P_{n,k}^{\left(a,b-1,c-1,d\right)  }\left(  x,y\right), \\
\mathcal{M}_{0,1}\mathcal{M}_{0,1}^{+}P_{n,k}^{\left(  a,b,c,d\right)}\left(  x,y\right)   &  =\left(  k+1\right)  \left(  k+b+c\right)P_{n,k}^{(a,b,c,d)}\left(  x,y\right), \\
\mathcal{M}_{0,2}^{+}\mathcal{M}_{0,2}P_{n,k}^{\left(  a,b+1,c-1,d\right)}\left(  x,y\right)   &  =\left(  k+c\right)  \left(  k+b+c+1\right)P_{n,k}^{\left(  a,b+1,c-1,d\right)  }\left(  x,y\right), \\
\mathcal{M}_{0,2}\mathcal{M}_{0,2}^{+}P_{n,k}^{\left(  a,b+1,c,d\right)}\left(  x,y\right)   &  =\left(  k+c\right)  \left(  k+b+c+1\right)P_{n,k}^{\left(  a,b+1,c,d\right)  }\left(  x,y\right), \\
\mathcal{M}_{0,3}^{+}\mathcal{M}_{0,3}P_{n,k}^{\left(  a,b-1,c+1,d\right)}\left(  x,y\right)   &  =\left(  k+b\right)  \left(  k+b+c+1\right)P_{n,k}^{\left(  a,b-1,c+1,d\right)  }\left(  x,y\right), \\
\mathcal{M}_{0,3}\mathcal{M}_{0,3}^{+}P_{n,k}^{\left(  a,b,c+1,d\right)}\left(  x,y\right)   &  =\left(  k+b\right)  \left(  k+b+c+1\right)P_{n,k}^{\left(  a,b,c+1,d\right)  }\left(  x,y\right), \\
\mathcal{M}_{0,4}^{+}\mathcal{M}_{0,4}P_{n,k-1}^{\left(  a,b+1,c,d\right)}\left(  x,y\right)   &  =k\left(  k+b+1\right)  P_{n,k-1}^{\left(a,b+1,c,d\right)  }\left(  x,y\right), \\
\mathcal{M}_{0,4}\mathcal{M}_{0,4}^{+}P_{n,k}^{\left(  a,b+1,c-1,d\right)}\left(  x,y\right)   &  =k\left(  k+b+1\right)  P_{n,k}^{\left(a,b+1,c-1,d\right)  }\left(  x,y\right), \\
\mathcal{M}_{0,5}^{+}\mathcal{M}_{0,5}P_{n,k-1}^{\left(  a,b,c+1,d\right)}\left(  x,y\right)   &  =k\left(  k+c+1\right)  P_{n,k-1}^{\left(a,b,c+1,d\right)  }\left(  x,y\right), \\
\mathcal{M}_{0,5}\mathcal{M}_{0,5}^{+}P_{n,k}^{\left(  a,b-1,c+1,d\right)}\left(  x,y\right)   &  =k\left(  k+c+1\right)  P_{n,k}^{\left(a,b-1,c+1,d\right)  }\left(  x,y\right), \\
\mathcal{M}_{0,6}^{+}\mathcal{M}_{0,6}P_{n,k}^{\left(  a,b,c-1,d\right)}\left(  x,y\right)   &  =\left(  k+b\right)  \left(  k+c\right)P_{n,k}^{\left(  a,b,c-1,d\right)  }\left(  x,y\right), \\
\mathcal{M}_{0,6}\mathcal{M}_{0,6}^{+}P_{n,k}^{\left(  a,b-1,c,d\right)}\left(  x,y\right)   &  =\left(  k+b\right)  \left(  k+c\right)P_{n,k}^{\left(  a,b-1,c,d\right)  }\left(  x,y\right),
\end{align*}
and
\begin{align*}
\mathcal{M}_{1,0}^{+}\mathcal{M}_{1,0}P_{n,k}^{\left(  a-1,b,c,d-1\right)}\left(  x,y\right)   &  =\left(  n-k\right)  \left(  n+k+a+b+c+d\right)P_{n,k}^{\left(  a-1,b,c,d-1\right)  }\left(  x,y\right), \\
\mathcal{M}_{1,0}\mathcal{M}_{1,0}^{+}P_{n,k}^{\left(  a,b,c,d\right)}\left(  x,y\right)   &  =\left(  n-k+1\right)  \left(  n+k+a+b+c+d+1\right)P_{n,k}^{(a,b,c,d)}\left(  x,y\right), \\
\mathcal{M}_{2,0}^{+}\mathcal{M}_{2,0}P_{n,k}^{\left(  a+1,b,c,d-1\right)}\left(  x,y\right)   &  =\left(  n+k+a+b+c+d+2\right) \\ &\times \left(  n+k+b+c+d+1\right)P_{n,k}^{\left(  a+1,b,c,d-1\right)  }\left(  x,y\right), \\
\mathcal{M}_{2,0}\mathcal{M}_{2,0}^{+}P_{n,k}^{\left(  a+1,b-1,c,d+1\right)}\left(  x,y\right)   &  =\left(  n+k+a+b+c+d+2\right) \\ &\times \left(  n+k+b+c+d+1\right)P_{n,k}^{\left(  a+1,b-1,c,d+1\right)  }\left(  x,y\right), \\
\mathcal{M}_{3,0}^{+}\mathcal{M}_{3,0}P_{n,k}^{\left(  a-1,b+1,c,d\right)}\left(  x,y\right)   &  =\left(  n+k+a+b+c+d+2\right)  \left(  n-k+a\right)P_{n,k}^{\left(  a-1,b+1,c,d\right)  }\left(  x,y\right), \\
\mathcal{M}_{3,0}\mathcal{M}_{3,0}^{+}P_{n,k}^{\left(  a,b+1,c,d\right)}\left(  x,y\right)   &  =\left(  n+k+a+b+c+d+2\right)  \left(  n-k+a\right)P_{n,k}^{\left(  a,b+1,c,d\right)  }\left(  x,y\right), \\
\mathcal{M}_{4,0}^{+}\mathcal{M}_{4,0}P_{n-1,k}^{\left(  a+1,b,c,d\right)}\left(  x,y\right)   &  =\left(  n-k\right)  \left(  n-k+a+1\right)P_{n-1,k}^{\left(  a+1,b,c,d\right)  }\left(  x,y\right), \\
\mathcal{M}_{4,0}\mathcal{M}_{4,0}^{+}P_{n,k}^{\left(  a+1,b-1,c,d\right)}\left(  x,y\right)   &  =\left(  n-k\right)  \left(  n-k+a+1\right)P_{n,k}^{\left(  a+1,b-1,c,d\right)  }\left(  x,y\right), \\
\mathcal{M}_{5,0}^{+}\mathcal{M}_{5,0}P_{n-1,k}^{\left(  a,b+1,c,d\right)}\left(  x,y\right)   &  =\left(  n-k\right)  \left(  n+k+b+c+d+2\right)P_{n-1,k}^{\left(  a,b+1,c,d\right)  }\left(  x,y\right), \\
\mathcal{M}_{5,0}\mathcal{M}_{5,0}^{+}P_{n,k}^{\left(  a-1,b+1,c,d\right)}\left(  x,y\right)   &  =\left(  n-k\right)  \left(  n+k+b+c+d+2\right)P_{n,k}^{\left(  a-1,b+1,c,d\right)  }\left(  x,y\right),
\end{align*}
\begin{align*}
\mathcal{M}_{6,0}^{+}\mathcal{M}_{6,0}P_{n,k}^{\left(  a,b,c,d-1\right)}\left(  x,y\right)   &  =\left(  n-k+a\right)  \left(  n+k+b+c+d+1\right)P_{n,k}^{\left(  a,b,c,d-1\right)  }\left(  x,y\right), \\
\mathcal{M}_{6,0}\mathcal{M}_{6,0}^{+}P_{n,k}^{\left(  a-1,b,c,d\right)}\left(  x,y\right)   &  =\left(  n-k+a\right)  \left(  n+k+b+c+d+1\right)P_{n,k}^{\left(  a-1,b,c,d\right)  }\left(  x,y\right).
\end{align*}

The polynomials $P_{n,k}^{(a,b,c,d)}\left(  x,y\right)  $ defined in \eqref{B} satisfy the following differential equations
\begin{equation}\label{A1}
L_{1}u:=y\left(  1-x-y\right)  \frac{\partial^{2}u}{\partial y^{2}}+\left(\left(  b+1\right)  (1-x)  -\left(  b+c+2\right)  y\right)\frac{\partial u}{\partial y}+k\left(  k+b+c+1\right)  u=0,
\end{equation}
and
\begin{multline}\label{C11}
L_{2}u:=x(1-x)  \dfrac{\partial^{2}u}{\partial x^{2}}-2xy\dfrac{\partial^{2}u}{\partial x\partial y}+y(1-y)\dfrac{\partial^{2}u}{\partial y^{2}}+\left(  a+1-\left(  a+b+c+d+3\right)x\right)  \dfrac{\partial u}{\partial x}\\
+\left(  b+1-\left(  a+b+c+d+3\right)  y+\dfrac{yd}{1-x}\right)\dfrac{\partial u}{\partial y}\\+\left(  n\left(  n+a+b+c+d+2\right)-\dfrac{kd}{1-x}\right)  u=0.
\end{multline}
The equation (\ref{A1}) is given in \cite{OTV}. From $L_{2}u-\frac{1}{1-x}L_{1}u=0$, it also yields
\begin{multline}\label{B1}
x(1-x)  \dfrac{\partial^{2}u}{\partial x^{2}}-2xy\dfrac{\partial^{2}u}{\partial x\partial y}+\dfrac{xy^{2}}{1-x}\dfrac{\partial^{2}u}{\partial y^{2}}+\left(  a+1-\left(  a+b+c+d+3\right)  x\right)\dfrac{\partial u}{\partial x}\\
-\left( \frac{y (a+1-x (a+b+c+d+3))}{1-x} \right)  \dfrac{\partial u}{\partial y}\\
+\left(  n\left(  n+a+b+c+d+2\right)  -\dfrac{k\left(  k+b+c+d+1\right)
}{1-x}\right)  u=0.
\end{multline}

\begin{remark}
For second order differential equations in the first set given above, the first, second and the last two equations give exactly the second order equation in the form of \eqref{A1}. The third, fourth, seventh and eighth equations give the equation \eqref{A1} multiplied by $y$. The fifth, sixth, ninth and tenth equations give the equation \eqref{A1} multiplied by $1-x-y$.

For second order differential equations in the second set given above, the first, second and the last two equations give exactly the second order equation in the form of \eqref{B1}. The third, fourth, seventh and eighth equations give the equation \eqref{B1} multiplied by $x.$ The fifth, sixth, ninth and tenth equations give the equation \eqref{B1} multiplied by $1-x$.
\end{remark}

\begin{remark}
The partial differential equation \eqref{B1} has also a monic solution \cite{MR2853206}, given by
\begin{multline}
\tilde{P}_{n,k}^{(a,b,c,d)}(x,y)=\left(  -1\right)  ^{n+k} \dfrac{\Gamma \left(a+n-k+1\right)  \Gamma \left(  a+b+c+d+n+k+2\right)  }{\Gamma \left(a+1\right)  \Gamma \left(  a+b+c+d+2n+2\right)  }\\
 \times y^{k} \, \hyper{2}{1}{-n+k,a+b+c+d+n+k+2}{a+1}{x},
\end{multline}
which can be also written as
\begin{equation}
\tilde{P}_{n,k}^{(a,b,c,d)}(x,y)=\dfrac{\left(  n-k\right)  !\Gamma \left(  a+b+c+d+n+k+2\right)}{\Gamma \left(  a+b+c+d+2n+2\right)  }y^{k}~\widetilde{P}_{n-k}^{\left(
b+c+d+2k+1,a\right)  }(x)
\end{equation}
in terms of shifted Jacobi polynomials $\widetilde{P}_{n}^{\left(  a,b\right)  }(x)$ defined in \eqref{eq:sjp}. Notice that in this representation of the monic solution of the partial differential equation \eqref{B1} the parameters of the shifted Jacobi polynomials depend on the degree i.e. they are varying orthogonal polynomials \cite{Zarzo}.
\end{remark}

\section{Differential relations for $P_{n}^{(\alpha,\beta,\gamma,\delta,a,b)  }(x,y,z)  $}\label{sec:3}

By using the 12 differential relations given in section \ref{sec:2} for the univariate shifted Jacobi polynomials $\widetilde{P}_{n}^{\left(  a,b\right)  }(x)$ defined in \eqref{eq:sjp} and the 24 differential relations also given in section \ref{sec:2} for the polynomials $P_{n,k}^{(a,b,c,d)}\left(x,y\right) $ defined in \eqref{B}, we obtain 36 differential relations for $P_{n_{1},n_{2},n_{3}}^{(\alpha,\beta,\gamma,\delta,a,b)}(x,y,z)$ defined in \eqref{C}, which map $P_{n_{1},n_{2},n_{3}}^{\left(\alpha,\beta,\gamma,\delta,a,b\right)  }(x,y,z)  $ to a scalar multiple of $P_{\widetilde{n_{1}},\widetilde{n_{2}},\widetilde{n_{3}}%
}^{\left(  \widetilde{\alpha},\widetilde{\beta},\widetilde{\gamma},\widetilde{\delta},\widetilde{a},\widetilde{b}\right)  }(x,y,z)$ where the parameters in the new family $P_{\widetilde{n_{1}},\widetilde{n_{2}},\widetilde{n_{3}}}^{\left(  \widetilde{\alpha},\widetilde{\beta},\widetilde{\gamma},\widetilde{\delta},\widetilde{a},\widetilde{b}\right)
}(x,y,z)  $ are $n_{1},n_{2},n_{3},\alpha,\beta,\gamma,\delta,a,b$ incremented or decremented by $0$ or $1$.

Let
\begin{align*}
\mathcal{N}_{0,1}u&=\dfrac{n_{3}u}{1-x-y}+\dfrac{\partial u}{\partial y}-\dfrac{z}{1-x-y}\dfrac{\partial u}{\partial z},\\
\mathcal{N}_{0,1}^{+}u&=\left(  y\left(  \gamma+\delta+n_{3}+b+1\right)-\left(  1-x-y\right)  \beta \right)  u-y\left(  1-x-y\right)  \dfrac{\partial u}{\partial y}+yz\dfrac{\partial u}{\partial z},\\
\mathcal{N}_{0,2}u&=\left(  n_{2}+2n_{3}+\beta+\gamma+\delta+b+2\right)u+\dfrac{n_{3}y}{1-x-y}u+y\dfrac{\partial u}{\partial y}-\dfrac{yz}{1-x-y}\dfrac{\partial u}{\partial z},\\
\mathcal{N}_{0,2}^{+}u&=\left(  n_{2}+2n_{3}+\gamma+\delta+b+1\right)u-\dfrac{\left(  n_{2}+n_{3}\right)  y}{1-x}u-\frac{y\left(  1-x-y\right)}{1-x}\dfrac{\partial u}{\partial y}+\frac{yz}{1-x}\dfrac{\partial u}{\partial z},\\
\mathcal{N}_{0,3}u&=\left(  n_{2}+n_{3}+\beta+\gamma+\delta+b+2\right)u-\left(  1-x-y\right)  \dfrac{\partial u}{\partial y}+z\dfrac{\partial u}{\partial z},\\
\mathcal{N}_{0,3}^{+}u&=\left(  \beta+\dfrac{n_{2}+n_{3}}{1-x}y\right)u+\dfrac{y\left(  1-x-y\right)  }{1-x}\dfrac{\partial u}{\partial y}
-\dfrac{yz}{1-x}\dfrac{\partial u}{\partial z},\\
\mathcal{N}_{0,4}u&=\left(  y\left(  n_{3}+\gamma+\delta+b+1\right)  -\left(1-x-y\right)  \left(  \beta+n_{2}+1\right)  \right)  u-y\left(  1-x-y\right)
\dfrac{\partial u}{\partial y}+yz\dfrac{\partial u}{\partial z},\\
\mathcal{N}_{0,4}^{+}u&=-\dfrac{n_{2}u}{1-x}+\dfrac{n_{3}y}{(1-x)\left(  1-x-y\right)  }u+\dfrac{y}{1-x}\dfrac{\partial u}{\partial y}-\dfrac{yz}{(1-x)  \left(  1-x-y\right)  }\dfrac{\partial u}{\partial z},\\
\mathcal{N}_{0,5}u&=\left(  y\left(  n_{2}+n_{3}+\gamma+\delta+b+2\right)-\left(  1-x-y\right)  \beta \right)  u-y\left(  1-x-y\right)  \dfrac{\partial u}{\partial y}+yz\dfrac{\partial u}{\partial z},\\
\mathcal{N}_{0,5}^{+}u&=\dfrac{n_{2}+n_{3}}{1-x}u+\left(  1-\dfrac{y}{1-x}\right)  \dfrac{\partial u}{\partial y}-\dfrac{z}{1-x}\dfrac{\partial u}{\partial z},\\
\mathcal{N}_{0,6}u&=\beta u+\dfrac{n_{3}y}{1-x-y}u+y\dfrac{\partial u}{\partial y}-\dfrac{yz}{1-x-y}\dfrac{\partial u}{\partial z},\\
\mathcal{N}_{0,6}^{+}u&=\left(  \gamma+\delta+n_{3}+b+1\right)  u-\left(1-x-y\right)  \dfrac{\partial u}{\partial y}+z\dfrac{\partial u}{\partial z},
\end{align*}
and
\begin{align*}
\mathcal{N}_{1,0}u&=\dfrac{n_{2}+n_{3}}{1-x}u+\dfrac{\partial u}{\partial x}-\dfrac{y}{1-x}\dfrac{\partial u}{\partial y}-\dfrac{z}{1-x}\dfrac{\partial u}{\partial z},\\
\mathcal{N}_{1,0}^{+}u&=\left(  x\left(  n_{2}+n_{3}+e+2\right)  -\alpha \right)  u-x(1-x)  \dfrac{\partial u}{\partial x}+xy\dfrac{\partial u}{\partial y}+xz\dfrac{\partial u}{\partial z},\\
\mathcal{N}_{2,0}u&=\left(  n+n_{2}+n_{3}+e+3\right)  u+\dfrac{\left(n_{2}+n_{3}\right)  xu}{1-x}+x\dfrac{\partial u}{\partial x}-\dfrac{xy}{1-x}\dfrac{\partial u}{\partial y}-\dfrac{xz}{1-x}\dfrac{\partial u}{\partial z},\\
\mathcal{N}_{2,0}^{+}u&=\left(  n+n_{2}+n_{3}+e-\alpha+2- nx\right)  u-x(1-x)  \dfrac{\partial u}{\partial x}+xy\dfrac{\partial u}{\partial y}+xz\dfrac{\partial u}{\partial
z},\\
\mathcal{N}_{3,0}u&=\left(  n+e+3\right)  u-(1-x)  \dfrac{\partial u}{\partial x}+y\dfrac{\partial u}{\partial y}+z\dfrac{\partial u}{\partial z},\\
\mathcal{N}_{3,0}^{+}u&=\left(  \alpha+nx\right)  u+x(1-x) \dfrac{\partial u}{\partial x}-xy\dfrac{\partial u}{\partial y}-xz\dfrac {\partial u}{\partial z},\\
\mathcal{N}_{4,0}u&=\left(  x\left(  n+e+3\right)  -\alpha-n_{1}-1\right)u-x(1-x)  \dfrac{\partial u}{\partial x}+xy\dfrac{\partial u}{\partial y}+xz\dfrac{\partial u}{\partial z},\\
\mathcal{N}_{4,0}^{+}u&=-nu+\dfrac{\left(  n_{2}+n_{3}\right)  u}{1-x}+x\dfrac{\partial u}{\partial x}-\dfrac{xy}{1-x}\dfrac{\partial u}{\partial y}-\dfrac{xz}{1-x}\dfrac{\partial u}{\partial z},  \\
\mathcal{N}_{5,0}u&=\left(  \left(  n+e+3\right)  x-\alpha \right)u-x(1-x)  \dfrac{\partial u}{\partial x}+xy\dfrac{\partial u}{\partial y}+xz\dfrac{\partial u}{\partial z}, \\
\mathcal{N}_{5,0}^{+}u&=nu+(1-x)  \dfrac{\partial u}{\partial x}-y\dfrac{\partial u}{\partial y}-z\dfrac{\partial u}{\partial z}, 
\end{align*}
\begin{align*}
\mathcal{N}_{6,0}u&=\alpha u+\dfrac{\left(  n_{2}+n_{3}\right)  xu}{1-x}+x\dfrac{\partial u}{\partial x}-\dfrac{xy}{1-x}\dfrac{\partial u}{\partial y}-\dfrac{xz}{1-x}\dfrac{\partial u}{\partial z},\\
\mathcal{N}_{6,0}^{+}u&=\left(  n_{2}+n_{3}+e-\alpha+2\right)  u-\left(1-x\right)  \dfrac{\partial u}{\partial x}+y\dfrac{\partial u}{\partial y}+z\dfrac{\partial u}{\partial z},
\end{align*}
as well as
\begin{align*}
\mathcal{O}_{1,0}u&=\dfrac{\partial u}{\partial z},\\
\mathcal{O}_{1,0}^{+}u&=\left(  z\delta-w\gamma \right)  u-zw\dfrac{\partial u}{\partial z},\\
\mathcal{O}_{2,0}u&=\left(  \delta+\gamma+n_{3}+1\right)  u+z\dfrac{\partial u}{\partial z},\\
\mathcal{O}_{2,0}^{+}u&=\left(  \delta+\dfrac{n_{3}w}{1-x-y}\right)u-\dfrac{zw}{1-x-y}\dfrac{\partial u}{\partial z},\\
\mathcal{O}_{3,0}u&=\left(  \delta+\gamma+n_{3}+1\right)  u-w\dfrac{\partial u}{\partial z},\\
\mathcal{O}_{3,0}^{+}u&=\left(  \gamma+\dfrac{n_{3}z}{1-x-y}\right)u+\dfrac{zw}{1-x-y}\dfrac{\partial u}{\partial z},\\
\mathcal{O}_{4,0}u&=\left(  z\delta-w\left(  \gamma+n_{3}+1\right)  \right)u-zw\dfrac{\partial u}{\partial z},\\
\mathcal{O}_{4,0}^{+}u&=-\dfrac{n_{3}u}{1-x-y}+\dfrac{z}{1-x-y}\dfrac{\partial u}{\partial z},\\
\mathcal{O}_{5,0}u&=\left(  \left(  \delta+n_{3}+1\right)  z-w\gamma \right) u-zw\dfrac{\partial u}{\partial z},\\
\mathcal{O}_{5,0}^{+}u&=\dfrac{n_{3}u}{1-x-y}+\dfrac{w}{1-x-y}\dfrac{\partial u}{\partial z},\\
\mathcal{O}_{6,0}u&=\gamma u+z\dfrac{\partial u}{\partial z},\\
\mathcal{O}_{6,0}^{+}u&=\delta u-w\dfrac{\partial u}{\partial z},
\end{align*}
where $n_{1}+n_{2}+n_{3}=n$, $\alpha+\beta+\gamma+\delta+a+b=e$ and $w=1-x-y-z$.

The differential relations above are chosen so that the recurrence relations in the next theorem are obtained for $\mathcal{N}_{i,0}$ (or $\mathcal{N}_{i,0}^{+})$ and $\mathcal{N}_{0,i}$ (or $\mathcal{N}_{0,i}^{+})$ by applying $\mathcal{M}_{i,0}$ (or $\mathcal{M}_{i,0}^{+}$) and $\mathcal{M}_{0,i}$ (or $\mathcal{M}_{0,i}^{+}$) to the first part
\begin{multline*}
P_{n_{1}+n_{2},n_{2}}^{\left( \alpha,\beta,\gamma+\delta+2n_{3}+b+1,a \right)}{(x,y)}\\
:=P_{n_{1}}^{\left(  \beta+\gamma+\delta+a+b+2n_{2}+2n_{3}+2,\alpha \right)}\left(  2x-1\right)  (1-x)  ^{n_{2}}P_{n_{2}}^{\left(\gamma+\delta+2n_{3}+b+1,\beta \right)  }\left(  \frac{2y}{1-x}-1\right)  ,
\end{multline*}
respectively, and for $\mathcal{O}_{i,0}$ (or $\mathcal{O}_{i,0}^{+})$ by applying $\widetilde{\mathcal{L}}_{i}$ (or $\widetilde{\mathcal{L}}_{i}^{+}$) to the second part $P_{n_{3}}^{\left(  \delta,\gamma \right)  }\left(  \frac {2z}{1-x-y}-1\right)  $ in the polynomial $P_{n_{1},n_{2},n_{3}}^{\left(\alpha,\beta,\gamma,\delta,a,b\right)  }(x,y,z)  $ for $1\leq i\leq6$.

\begin{theorem}\label{Theorem 1} The following sparse recurrence relations for $P_{n_{1},n_{2},n_{3}}^{(\alpha,\beta,\gamma,\delta,a,b)  }(x,y,z)$  hold true
\begin{align*}
\mathcal{N}_{0,1}P_{n_{1},n_{2},n_{3}}^{(\alpha,\beta,\gamma,\delta,a,b)  }(x,y,z)   &  =\left(  n_{2}+2n_{3}+\beta+\gamma+\delta+b+2\right)  P_{n_{1},n_{2}-1,n_{3}}^{\left(  \alpha,\beta+1,\gamma,\delta,a,b+1\right)  }(x,y,z), \\
\mathcal{N}_{0,1}^{+}P_{n_{1},n_{2},n_{3}}^{(\alpha,\beta,\gamma,\delta,a,b)  }(x,y,z)   &  =\left(  n_{2}+1\right)P_{n_{1},n_{2}+1,n_{3}}^{\left(  \alpha,\beta-1,\gamma,\delta,a,b-1\right)}(x,y,z),\\
\mathcal{N}_{0,2}P_{n_{1},n_{2},n_{3}}^{(\alpha,\beta,\gamma,\delta,a,b)  }(x,y,z)   &  =\left(  n_{2}+2n_{3}+\beta+\gamma+\delta+b+2\right)  P_{n_{1},n_{2},n_{3}}^{\left(  \alpha,\beta,\gamma,\delta,a-1,b+1\right)  }(x,y,z), \\
\mathcal{N}_{0,2}^{+}P_{n_{1},n_{2},n_{3}}^{(\alpha,\beta,\gamma,\delta,a,b)  }(x,y,z)   &  =\left(  n_{2}+2n_{3}+\gamma+\delta+b+1\right)  P_{n_{1},n_{2},n_{3}}^{\left(  \alpha,\beta,\gamma,\delta,a+1,b-1\right)  }(x,y,z),
\end{align*}
\begin{align*}
\mathcal{N}_{0,3}P_{n_{1},n_{2},n_{3}}^{(\alpha,\beta,\gamma,\delta,a,b)  }(x,y,z)   &  =\left(  n_{2}+2n_{3}+\beta+\gamma+\delta+b+2\right)  P_{n_{1},n_{2},n_{3}}^{\left(  \alpha,\beta+1,\gamma,\delta,a-1,b\right)  }(x,y,z), \\
\mathcal{N}_{0,3}^{+}P_{n_{1},n_{2},n_{3}}^{(\alpha,\beta,\gamma,\delta,a,b)  }(x,y,z)   &  =\left(  n_{2}+\beta \right)
P_{n_{1},n_{2},n_{3}}^{\left(  \alpha,\beta-1,\gamma,\delta,a+1,b\right)}(x,y,z), \\
\mathcal{N}_{0,4}P_{n_{1},n_{2},n_{3}}^{(\alpha,\beta,\gamma,\delta,a,b)  }(x,y,z)   &  =\left(  n_{2}+1\right)P_{n_{1},n_{2}+1,n_{3}}^{\left(  \alpha,\beta,\gamma,\delta,a-1,b-1\right)}(x,y,z), \\
\mathcal{N}_{0,4}^{+}P_{n_{1},n_{2},n_{3}}^{(\alpha,\beta,\gamma,\delta,a,b)  }(x,y,z)   &  =\left(  n_{2}+\beta \right)
P_{n_{1},n_{2}-1,n_{3}}^{\left(  \alpha,\beta,\gamma,\delta,a+1,b+1\right)}(x,y,z), \\
\mathcal{N}_{0,5}P_{n_{1},n_{2},n_{3}}^{(\alpha,\beta,\gamma,\delta,a,b)  }(x,y,z)   &  =\left(  n_{2}+1\right)P_{n_{1},n_{2}+1,n_{3}}^{\left(  \alpha,\beta-1,\gamma,\delta,a-1,b\right)}(x,y,z), \\
\mathcal{N}_{0,5}^{+}P_{n_{1},n_{2},n_{3}}^{(\alpha,\beta,\gamma,\delta,a,b)  }(x,y,z)   &  =\left(  n_{2}+2n_{3}+\gamma+\delta+b+1\right)  P_{n_{1},n_{2}-1,n_{3}}^{\left(  \alpha,\beta+1,\gamma,\delta,a+1,b\right)  }(x,y,z),\\
\mathcal{N}_{0,6}P_{n_{1},n_{2},n_{3}}^{(\alpha,\beta,\gamma,\delta,a,b)  }(x,y,z)   &  =\left(  n_{2}+\beta \right)P_{n_{1},n_{2},n_{3}}^{\left(  \alpha,\beta-1,\gamma,\delta,a,b+1\right)}(x,y,z), \\
\mathcal{N}_{0,6}^{+}P_{n_{1},n_{2},n_{3}}^{(\alpha,\beta,\gamma,\delta,a,b)  }(x,y,z)   &  =\left(  n_{2}+2n_{3}+\gamma+\delta+b+1\right)  P_{n_{1},n_{2},n_{3}}^{\left(  \alpha,\beta+1,\gamma,\delta,a,b-1\right)  }(x,y,z).
\end{align*}
Moreover,
\begin{align*}
\mathcal{N}_{1,0}P_{n_{1},n_{2},n_{3}}^{(\alpha,\beta,\gamma,\delta,a,b)  }(x,y,z)   &  =\left(  n+n_{2}+n_{3}+e+3\right)  P_{n_{1}-1,n_{2},n_{3}}^{\left(  \alpha+1,\beta,\gamma,\delta,a+1,b\right)  }(x,y,z), \\
\mathcal{N}_{1,0}^{+}P_{n_{1},n_{2},n_{3}}^{(\alpha,\beta,\gamma,\delta,a,b)  }(x,y,z)   &  =\left(  n_{1}+1\right)P_{n_{1}+1,n_{2},n_{3}}^{\left(  \alpha-1,\beta,\gamma,\delta,a-1,b\right)}(x,y,z), \\
\mathcal{N}_{2,0}P_{n_{1},n_{2},n_{3}}^{(\alpha,\beta,\gamma,\delta,a,b)  }(x,y,z)   &  =\left(  n+n_{2}+n_{3}+e+3\right)  P_{n_{1},n_{2},n_{3}}^{\left(  \alpha,\beta,\gamma,\delta,a+1,b\right)  }(x,y,z), \\
\mathcal{N}_{2,0}^{+}P_{n_{1},n_{2},n_{3}}^{(\alpha,\beta,\gamma,\delta,a,b)  }(x,y,z)   &  =\left(  n+n_{2}+n_{3}+e-\alpha+2\right)  P_{n_{1},n_{2},n_{3}}^{\left(  \alpha,\beta,\gamma,\delta,a-1,b\right)  }(x,y,z), \\
\mathcal{N}_{3,0}P_{n_{1},n_{2},n_{3}}^{(\alpha,\beta,\gamma,\delta,a,b)  }(x,y,z)   &  =\left(  n+n_{2}+n_{3}+e+3\right)  P_{n_{1},n_{2},n_{3}}^{\left(  \alpha+1,\beta,\gamma,\delta,a,b\right)  }(x,y,z), \\
\mathcal{N}_{3,0}^{+}P_{n_{1},n_{2},n_{3}}^{(\alpha,\beta,\gamma,\delta,a,b)  }(x,y,z)   &  =\left(  n_{1}+\alpha \right)P_{n_{1},n_{2},n_{3}}^{\left(  \alpha-1,\beta,\gamma,\delta,a,b\right)}(x,y,z), \\
\mathcal{N}_{4,0}P_{n_{1},n_{2},n_{3}}^{(\alpha,\beta,\gamma,\delta,a,b)  }(x,y,z)   &  =\left(  n_{1}+1\right)P_{n_{1}+1,n_{2},n_{3}}^{\left(  \alpha,\beta,\gamma,\delta,a-1,b\right)}(x,y,z), \\
\mathcal{N}_{4,0}^{+}P_{n_{1},n_{2},n_{3}}^{(\alpha,\beta,\gamma,\delta,a,b)  }(x,y,z)   &  =\left(  n_{1}+\alpha \right)P_{n_{1}-1,n_{2},n_{3}}^{\left(  \alpha,\beta,\gamma,\delta,a+1,b\right)}(x,y,z), \\
\mathcal{N}_{5,0}P_{n_{1},n_{2},n_{3}}^{(\alpha,\beta,\gamma,\delta,a,b)  }(x,y,z)   &  =\left(  n_{1}+1\right)P_{n_{1}+1,n_{2},n_{3}}^{\left(  \alpha-1,\beta,\gamma,\delta,a,b\right)}(x,y,z), \\
\mathcal{N}_{5,0}^{+}P_{n_{1},n_{2},n_{3}}^{(\alpha,\beta,\gamma,\delta,a,b)  }(x,y,z)   &  =\left(  n+n_{2}+n_{3}+e-\alpha+2\right)  P_{n_{1}-1,n_{2},n_{3}}^{\left(  \alpha+1,\beta,\gamma,\delta,a,b\right)  }(x,y,z), \\
\mathcal{N}_{6,0}P_{n_{1},n_{2},n_{3}}^{(\alpha,\beta,\gamma,\delta,a,b)  }(x,y,z)   &  =\left(  n_{1}+\alpha \right)P_{n_{1},n_{2},n_{3}}^{\left(  \alpha-1,\beta,\gamma,\delta,a+1,b\right)}(x,y,z), \\
\mathcal{N}_{6,0}^{+}P_{n_{1},n_{2},n_{3}}^{(\alpha,\beta,\gamma,\delta,a,b)  }(x,y,z)   &  =\left(  n+n_{2}+n_{3}+e-\alpha+2\right)  P_{n_{1},n_{2},n_{3}}^{\left(  \alpha+1,\beta,\gamma,\delta,a-1,b\right)  }(x,y,z).
\end{align*}
Also,
\begin{align*}
\mathcal{O}_{1,0}P_{n_{1},n_{2},n_{3}}^{(\alpha,\beta,\gamma,\delta,a,b)  }(x,y,z)   &  =\left(  n_{3}+\delta+\gamma+1\right)  P_{n_{1},n_{2},n_{3}-1}^{\left(  \alpha,\beta,\gamma+1,\delta+1,a,b\right)  }(x,y,z), \\
\mathcal{O}_{1,0}^{+}P_{n_{1},n_{2},n_{3}}^{(\alpha,\beta,\gamma,\delta,a,b)  }(x,y,z)   &  =\left(  n_{3}+1\right)P_{n_{1},n_{2},n_{3}+1}^{\left(  \alpha,\beta,\gamma-1,\delta-1,a,b\right)}(x,y,z), \\
\mathcal{O}_{2,0}P_{n_{1},n_{2},n_{3}}^{(\alpha,\beta,\gamma,\delta,a,b)  }(x,y,z)   &  =\left(  n_{3}+\delta+\gamma+1\right)  P_{n_{1},n_{2},n_{3}}^{\left(  \alpha,\beta,\gamma,\delta+1,a,b-1\right)  }(x,y,z), \\
\mathcal{O}_{2,0}^{+}P_{n_{1},n_{2},n_{3}}^{(\alpha,\beta,\gamma,\delta,a,b)  }(x,y,z)   &  =\left(  n_{3}+\delta \right)P_{n_{1},n_{2},n_{3}}^{\left(  \alpha,\beta,\gamma,\delta-1,a,b+1\right)}(x,y,z), \\
\mathcal{O}_{3,0}P_{n_{1},n_{2},n_{3}}^{(\alpha,\beta,\gamma,\delta,a,b)  }(x,y,z)   &  =\left(  n_{3}+\delta+\gamma+1\right)  P_{n_{1},n_{2},n_{3}}^{\left(  \alpha,\beta,\gamma+1,\delta,a,b-1\right)  }(x,y,z), \\
\mathcal{O}_{3,0}^{+}P_{n_{1},n_{2},n_{3}}^{(\alpha,\beta,\gamma,\delta,a,b)  }(x,y,z)   &  =\left(  n_{3}+\gamma \right)P_{n_{1},n_{2},n_{3}}^{\left(  \alpha,\beta,\gamma-1,\delta,a,b+1\right)}(x,y,z), \\
\mathcal{O}_{4,0}P_{n_{1},n_{2},n_{3}}^{(\alpha,\beta,\gamma,\delta,a,b)  }(x,y,z)   &  =\left(  n_{3}+1\right)P_{n_{1},n_{2},n_{3}+1}^{\left(  \alpha,\beta,\gamma,\delta-1,a,b-1\right)}(x,y,z), \\
\mathcal{O}_{4,0}^{+}P_{n_{1},n_{2},n_{3}}^{(\alpha,\beta,\gamma,\delta,a,b)  }(x,y,z)   &  =\left(  n_{3}+\gamma \right)P_{n_{1},n_{2},n_{3}-1}^{\left(  \alpha,\beta,\gamma,\delta+1,a,b+1\right)}(x,y,z), \\
\mathcal{O}_{5,0}P_{n_{1},n_{2},n_{3}}^{(\alpha,\beta,\gamma,\delta,a,b)  }(x,y,z)   &  =\left(  n_{3}+1\right)P_{n_{1},n_{2},n_{3}+1}^{\left(  \alpha,\beta,\gamma-1,\delta,a,b-1\right)}(x,y,z), \\
\mathcal{O}_{5,0}^{+}P_{n_{1},n_{2},n_{3}}^{(\alpha,\beta,\gamma,\delta,a,b)  }(x,y,z)   &  =\left(  n_{3}+\delta \right)P_{n_{1},n_{2},n_{3}-1}^{\left(  \alpha,\beta,\gamma+1,\delta,a,b+1\right)}(x,y,z), \\
\mathcal{O}_{6,0}P_{n_{1},n_{2},n_{3}}^{(\alpha,\beta,\gamma,\delta,a,b)  }(x,y,z)   &  =\left(  n_{3}+\gamma \right)P_{n_{1},n_{2},n_{3}}^{\left(  \alpha,\beta,\gamma-1,\delta+1,a,b\right)}(x,y,z), \\
\mathcal{O}_{6,0}^{+}P_{n_{1},n_{2},n_{3}}^{(\alpha,\beta,\gamma,\delta,a,b)  }(x,y,z)   &  =\left(  n_{3}+\delta \right)P_{n_{1},n_{2},n_{3}}^{\left(  \alpha,\beta,\gamma+1,\delta-1,a,b\right)}(x,y,z),
\end{align*}
where $n_{1}+n_{2}+n_{3}=n$ and $\alpha+\beta+\gamma+\delta+a+b=e$.
\end{theorem}

\begin{proof}
Let us denote
\[
P_{n_{1},n_{2},n_{3}}^{(\alpha,\beta,\gamma,\delta,a,b)  }(x,y,z)  :=P_{n_{1}+n_{2},n_{2}}^{\left(  \alpha,\beta,\gamma+\delta+2n_{3}+b+1,a\right)  }\left(  x,y\right)  \left(  1-x-y\right)^{n_{3}}P_{n_{3}}^{\left(  \delta,\gamma \right)  }\left(  \frac{2z}{1-x-y}-1\right).
\]
If we apply $\mathcal{M}_{0,1}$ to the both sides of $P_{n_{1},n_{2},n_{3}}^{(\alpha,\beta,\gamma,\delta,a,b)  }(x,y,z)$, we have
\[
\mathcal{N}_{0,1}P_{n_{1},n_{2},n_{3}}^{(\alpha,\beta,\gamma,\delta,a,b)  }(x,y,z)  =\left(  1-x-y\right)  ^{n_{3}}P_{n_{3}}^{\left(  \delta,\gamma \right)  }\left(  \frac{2z}{1-x-y}-1\right)\mathcal{M}_{0,1}P_{n_{1}+n_{2},n_{2}}^{\left(  \alpha,\beta,\gamma+\delta+2n_{3}+b+1,a\right)  }\left(  x,y\right)  .
\]
From the relation
\[
\mathcal{M}_{0,1}P_{n,k}^{(a,b,c,d)}\left(  x,y\right)=\left(  k+b+c+1\right)  P_{n-1,k-1}^{\left(  a,b+1,c+1,d\right)  }\left(x,y\right),
\]
it follows
\[
\mathcal{N}_{0,1}P_{n_{1},n_{2},n_{3}}^{(\alpha,\beta,\gamma,\delta,a,b)  }(x,y,z)  =\left(  n_{2}+2n_{3}+\beta+\gamma+\delta+b+2\right)  P_{n_{1},n_{2}-1,n_{3}}^{\left(  \alpha,\beta+1,\gamma,\delta,a,b+1\right)  }(x,y,z).
\]
Similarly, in view of $\mathcal{M}_{i,0},\mathcal{M}_{i,0}^{+},\mathcal{M}_{0,i},\mathcal{M}_{0,i}^{+},\widetilde{\mathcal{L}}_{i}$ and$\  \widetilde {\mathcal{L}}_{i}^{+}, 1\leq i\leq6,$ the other recurrence relations are obtained.
\end{proof}
{}From these relations we have
\begin{align*}
\mathcal{N}_{0,1}^{+}\mathcal{N}_{0,1}P_{n_{1},n_{2},n_{3}}^{\left(\alpha,\beta-1,\gamma,\delta,a,b-1\right)  }(x,y,z)   &=n_{2}\left(  n_{2}+2n_{3}+\beta+\gamma+\delta+b\right)  P_{n_{1},n_{2},n_{3}}^{\left(  \alpha,\beta-1,\gamma,\delta,a,b-1\right)  }(x,y,z),\\
\mathcal{N}_{0,1}\mathcal{N}_{0,1}^{+}P_{n_{1},n_{2},n_{3}}^{\left(\alpha,\beta,\gamma,\delta,a,b\right)  }(x,y,z)   &  =\left(n_{2}+1\right)  \left(  n_{2}+2n_{3}+\beta+\gamma+\delta+b+1\right)P_{n_{1},n_{2},n_{3}}^{(\alpha,\beta,\gamma,\delta,a,b)  }(x,y,z),  \\
\mathcal{N}_{0,2}^{+}\mathcal{N}_{0,2}P_{n_{1},n_{2},n_{3}}^{\left(\alpha,\beta+1,\gamma,\delta,a,b-1\right)  }(x,y,z)  & =\left(n_{2}+2n_{3}+\beta+\gamma+\delta+b+2\right)  \left(  n_{2}+2n_{3}+\gamma+\delta+b+1\right) \\ & \times  P_{n_{1},n_{2},n_{3}}^{\left(  \alpha,\beta+1,\gamma,\delta,a,b-1\right)  }(x,y,z),  \\
\mathcal{N}_{0,2}\mathcal{N}_{0,2}^{+}P_{n_{1},n_{2},n_{3}}^{\left(\alpha,\beta+1,\gamma,\delta,a,b\right)  }(x,y,z)   &  =\left(n_{2}+2n_{3}+\beta+\gamma+\delta+b+2\right)  \left(  n_{2}+2n_{3}+\gamma+\delta+b+1\right) \\ & \times P_{n_{1},n_{2},n_{3}}^{\left(  \alpha,\beta+1,\gamma,\delta,a,b\right)  }(x,y,z),  \\
\mathcal{N}_{0,3}^{+}\mathcal{N}_{0,3}P_{n_{1},n_{2},n_{3}}^{\left(\alpha,\beta-1,\gamma,\delta,a,b+1\right)  }(x,y,z)   &  =\left(n_{2}+\beta \right)  \left(  n_{2}+2n_{3}+\beta+\gamma+\delta+b+2\right)P_{n_{1},n_{2},n_{3}}^{\left(  \alpha,\beta-1,\gamma,\delta,a,b+1\right)}(x,y,z),  \\
\mathcal{N}_{0,3}\mathcal{N}_{0,3}^{+}P_{n_{1},n_{2},n_{3}}^{\left(\alpha,\beta,\gamma,\delta,a,b+1\right)  }(x,y,z)   &  =\left(n_{2}+\beta \right)  \left(  n_{2}+2n_{3}+\beta+\gamma+\delta+b+2\right)P_{n_{1},n_{2},n_{3}}^{\left(  \alpha,\beta,\gamma,\delta,a,b+1\right)}(x,y,z),  \\
\mathcal{N}_{0,4}^{+}\mathcal{N}_{0,4}P_{n_{1},n_{2}-1,n_{3}}^{\left(\alpha,\beta+1,\gamma,\delta,a,b\right)  }(x,y,z)   &=n_{2}\left(  n_{2}+\beta+1\right)  P_{n_{1},n_{2}-1,n_{3}}^{\left(\alpha,\beta+1,\gamma,\delta,a,b\right)  }(x,y,z),  \\
\mathcal{N}_{0,4}\mathcal{N}_{0,4}^{+}P_{n_{1},n_{2},n_{3}}^{\left(\alpha,\beta+1,\gamma,\delta,a,b-1\right)  }(x,y,z)   &=n_{2}\left(  n_{2}+\beta+1\right)  P_{n_{1},n_{2},n_{3}}^{\left(\alpha,\beta+1,\gamma,\delta,a,b-1\right)  }(x,y,z),  \\
\mathcal{N}_{0,5}^{+}\mathcal{N}_{0,5}P_{n_{1},n_{2}-1,n_{3}}^{\left(\alpha,\beta,\gamma,\delta,a,b+1\right)  }(x,y,z)   &=n_{2}\left(  n_{2}+2n_{3}+\gamma+\delta+b+2\right)  P_{n_{1},n_{2}-1,n_{3}}^{\left(  \alpha,\beta,\gamma,\delta,a,b+1\right)  }(x,y,z), \\
\mathcal{N}_{0,5}\mathcal{N}_{0,5}^{+}P_{n_{1},n_{2},n_{3}}^{\left(\alpha,\beta-1,\gamma,\delta,a,b+1\right)  }(x,y,z)   &=n_{2}\left(  n_{2}+2n_{3}+\gamma+\delta+b+2\right)  P_{n_{1},n_{2},n_{3}}^{\left(  \alpha,\beta-1,\gamma,\delta,a,b+1\right)  }(x,y,z),\\
\mathcal{N}_{0,6}^{+}\mathcal{N}_{0,6}P_{n_{1},n_{2},n_{3}}^{\left(\alpha,\beta,\gamma,\delta,a,b-1\right)  }(x,y,z)   &  =\left(n_{2}+\beta\right)  \left(  n_{2}+2n_{3}+\gamma+\delta+b+1\right)P_{n_{1},n_{2},n_{3}}^{\left(  \alpha,\beta,\gamma,\delta,a,b-1\right)}(x,y,z),  \\
\mathcal{N}_{0,6}\mathcal{N}_{0,6}^{+}P_{n_{1},n_{2},n_{3}}^{\left(\alpha,\beta-1,\gamma,\delta,a,b\right)  }(x,y,z)   &  =\left(n_{2}+\beta\right)  \left(  n_{2}+2n_{3}+\gamma+\delta+b+1\right)P_{n_{1},n_{2},n_{3}}^{\left(  \alpha,\beta-1,\gamma,\delta,a,b\right)}(x,y,z),
\end{align*}
and
\begin{align*}
\mathcal{N}_{1,0}^{+}\mathcal{N}_{1,0}P_{n_{1},n_{2},n_{3}}^{\left(\alpha-1,\beta,\gamma,\delta,a,b-1\right)  }(x,y,z)   &=n_{1}\left(  n+n_{2}+n_{3}+e+1\right)  P_{n_{1},n_{2},n_{3}}^{\left(\alpha-1,\beta,\gamma,\delta,a,b-1\right)  }(x,y,z),  \\
\mathcal{N}_{1,0}\mathcal{N}_{1,0}^{+}P_{n_{1},n_{2},n_{3}}^{\left(\alpha,\beta,\gamma,\delta,a,b\right)  }(x,y,z)   &  =\left(n_{1}+1\right)  \left(  n+n_{2}+n_{3}+e+2\right)  P_{n_{1},n_{2},n_{3}}^{(\alpha,\beta,\gamma,\delta,a,b)  }(x,y,z),   \\
\mathcal{N}_{2,0}^{+}\mathcal{N}_{2,0}P_{n_{1},n_{2},n_{3}}^{\left(\alpha+1,\beta,\gamma,\delta,a,b-1\right)  }(x,y,z)   &  =\left(n+n_{2}+n_{3}+e+3\right)  \left(  n+n_{2}+n_{3}+e-\alpha+2\right) \\ & \times P_{n_{1},n_{2},n_{3}}^{\left(  \alpha+1,\beta,\gamma,\delta,a,b-1\right)}(x,y,z),
\end{align*}
\begin{align*}
\mathcal{N}_{2,0}\mathcal{N}_{2,0}^{+}P_{n_{1},n_{2},n_{3}}^{\left(\alpha+1,\beta,\gamma,\delta,a,b\right)  }(x,y,z)   &  =\left(n+n_{2}+n_{3}+e+3\right)  \left(  n+n_{2}+n_{3}+e-\alpha+2\right) \\ & \times P_{n_{1},n_{2},n_{3}}^{\left(  \alpha+1,\beta,\gamma,\delta,a,b\right)}(x,y,z),  \\
\mathcal{N}_{3,0}^{+}\mathcal{N}_{3,0}P_{n_{1},n_{2},n_{3}}^{\left(\alpha-1,\beta,\gamma,\delta,a+1,b\right)  }(x,y,z)   &  =\left(n_{1}+\alpha \right)  \left(  n+n_{2}+n_{3}+e+3\right)  P_{n_{1},n_{2},n_{3}}^{\left(  \alpha-1,\beta,\gamma,\delta,a+1,b\right)  }(x,y,z),\\
\mathcal{N}_{3,0}\mathcal{N}_{3,0}^{+}P_{n_{1},n_{2},n_{3}}^{\left(\alpha,\beta,\gamma,\delta,a+1,b\right)  }(x,y,z)   &  =\left(n_{1}+\alpha \right)  \left(  n+n_{2}+n_{3}+e+3\right)  P_{n_{1},n_{2},n_{3}}^{\left(  \alpha,\beta,\gamma,\delta,a+1,b\right)  }(x,y,z),  \\
\mathcal{N}_{4,0}^{+}\mathcal{N}_{4,0}P_{n_{1}-1,n_{2},n_{3}}^{\left(\alpha+1,\beta,\gamma,\delta,a,b\right)  }(x,y,z)   &=n_{1}\left(  n_{1}+\alpha+1\right)  P_{n_{1}-1,n_{2},n_{3}}^{\left(\alpha+1,\beta,\gamma,\delta,a,b\right)  }(x,y,z) , \\
\mathcal{N}_{4,0}\mathcal{N}_{4,0}^{+}P_{n_{1},n_{2},n_{3}}^{\left(\alpha+1,\beta,\gamma,\delta,a,b-1\right)  }(x,y,z)   &=n_{1}\left(  n_{1}+\alpha+1\right)  P_{n_{1},n_{2},n_{3}}^{\left(\alpha+1,\beta,\gamma,\delta,a,b-1\right)  }(x,y,z),  \\
\mathcal{N}_{5,0}^{+}\mathcal{N}_{5,0}P_{n_{1}-1,n_{2},n_{3}}^{\left(\alpha,\beta,\gamma,\delta,a+1,b\right)  }(x,y,z)   &  =n_{1}  \left(  n+n_{2}+n_{3}+e-\alpha+3\right)  P_{n_{1}-1,n_{2},n_{3}}^{\left(  \alpha,\beta,\gamma,\delta,a+1,b\right)  }(x,y,z),\\
\mathcal{N}_{5,0}\mathcal{N}_{5,0}^{+}P_{n_{1},n_{2},n_{3}}^{\left(\alpha-1,\beta,\gamma,\delta,a+1,b\right)  }(x,y,z)   &  =n_{1}  \left(  n+n_{2}+n_{3}+e-\alpha+3\right)  P_{n_{1},n_{2},n_{3}}^{\left(  \alpha-1,\beta,\gamma,\delta,a+1,b\right)  }(x,y,z),  \\
\mathcal{N}_{6,0}^{+}\mathcal{N}_{6,0}P_{n_{1},n_{2},n_{3}}^{\left(\alpha,\beta,\gamma,\delta,a,b-1\right)  }(x,y,z)   &  =\left(n_{1}+\alpha \right)  \left(  n+n_{2}+n_{3}+e-\alpha+2\right)  P_{n_{1},n_{2},n_{3}}^{\left(  \alpha,\beta,\gamma,\delta,a,b-1\right)  }(x,y,z),  \\
\mathcal{N}_{6,0}\mathcal{N}_{6,0}^{+}P_{n_{1},n_{2},n_{3}}^{\left(\alpha-1,\beta,\gamma,\delta,a,b\right)  }(x,y,z)   &  =\left(n_{1}+\alpha \right)  \left(  n+n_{2}+n_{3}+e-\alpha+2\right)  P_{n_{1},n_{2},n_{3}}^{\left(  \alpha-1,\beta,\gamma,\delta,a,b\right)  }(x,y,z),
\end{align*}
as well as
\begin{align*}
\mathcal{O}_{1,0}^{+}\mathcal{O}_{1,0}P_{n_{1},n_{2},n_{3}}^{\left(\alpha,\beta,\gamma-1,\delta-1,a,b\right)  }(x,y,z)   &=n_{3}\left(  n_{3}+\gamma+\delta-1\right)  P_{n_{1},n_{2},n_{3}}^{\left(\alpha,\beta,\gamma-1,\delta-1,a,b\right)  }(x,y,z),  \\
\mathcal{O}_{1,0}\mathcal{O}_{1,0}^{+}P_{n_{1},n_{2},n_{3}}^{\left(\alpha,\beta,\gamma,\delta,a,b\right)  }(x,y,z)   &  =\left(n_{3}+1\right)  \left(  n_{3}+\gamma+\delta \right)  P_{n_{1},n_{2},n_{3}}^{(\alpha,\beta,\gamma,\delta,a,b)  }(x,y,z),  \\
\mathcal{O}_{2,0}^{+}\mathcal{O}_{2,0}P_{n_{1},n_{2},n_{3}}^{\left(\alpha,\beta,\gamma+1,\delta-1,a,b\right)  }(x,y,z)   &  =\left(n_{3}+\delta \right)  \left(  n_{3}+\gamma+\delta+1\right)  P_{n_{1},n_{2},n_{3}}^{\left(  \alpha,\beta,\gamma+1,\delta-1,a,b\right)  }(x,y,z),  \\
\mathcal{O}_{2,0}\mathcal{O}_{2,0}^{+}P_{n_{1},n_{2},n_{3}}^{\left(\alpha,\beta,\gamma+1,\delta,a,b\right)  }(x,y,z)   &  =\left(n_{3}+\delta \right)  \left(  n_{3}+\gamma+\delta+1\right)  P_{n_{1},n_{2},n_{3}}^{\left(  \alpha,\beta,\gamma+1,\delta,a,b\right)  }(x,y,z),  \\
\mathcal{O}_{3,0}^{+}\mathcal{O}_{3,0}P_{n_{1},n_{2},n_{3}}^{\left(\alpha,\beta,\gamma-1,\delta+1,a,b\right)  }(x,y,z)   &  =\left(n_{3}+\gamma \right)  \left(  n_{3}+\gamma+\delta+1\right)  P_{n_{1},n_{2},n_{3}}^{\left(  \alpha,\beta,\gamma-1,\delta+1,a,b\right)  }(x,y,z),  \\
\mathcal{O}_{3,0}\mathcal{O}_{3,0}^{+}P_{n_{1},n_{2},n_{3}}^{\left(\alpha,\beta,\gamma,\delta+1,a,b\right)  }(x,y,z)   &  =\left(n_{3}+\gamma \right)  \left(  n_{3}+\gamma+\delta+1\right)  P_{n_{1},n_{2},n_{3}}^{\left(  \alpha,\beta,\gamma,\delta+1,a,b\right)  }(x,y,z),  \\
\mathcal{O}_{4,0}^{+}\mathcal{O}_{4,0}P_{n_{1},n_{2},n_{3}-1}^{\left(\alpha,\beta,\gamma+1,\delta,a,b\right)  }(x,y,z)   &=n_{3}\left(  n_{3}+\gamma+1\right)  P_{n_{1},n_{2},n_{3}-1}^{\left(\alpha,\beta,\gamma+1,\delta,a,b\right)  }(x,y,z),  \\
\mathcal{O}_{4,0}\mathcal{O}_{4,0}^{+}P_{n_{1},n_{2},n_{3}}^{\left(\alpha,\beta,\gamma+1,\delta-1,a,b\right)  }(x,y,z)   &=n_{3}\left(  n_{3}+\gamma+1\right)  P_{n_{1},n_{2},n_{3}}^{\left(\alpha,\beta,\gamma+1,\delta-1,a,b\right)  }(x,y,z),  \\
\mathcal{O}_{5,0}^{+}\mathcal{O}_{5,0}P_{n_{1},n_{2},n_{3}-1}^{\left(\alpha,\beta,\gamma,\delta+1,a,b\right)  }(x,y,z)   &=n_{3}\left(  n_{3}+\delta+1\right)  P_{n_{1},n_{2},n_{3}-1}^{\left(\alpha,\beta,\gamma,\delta+1,a,b\right)  }(x,y,z),  \\
\mathcal{O}_{5,0}\mathcal{O}_{5,0}^{+}P_{n_{1},n_{2},n_{3}}^{\left(\alpha,\beta,\gamma-1,\delta+1,a,b\right)  }(x,y,z)   &=n_{3}\left(  n_{3}+\delta+1\right)  P_{n_{1},n_{2},n_{3}}^{\left(\alpha,\beta,\gamma-1,\delta+1,a,b\right)  }(x,y,z),  \\
\mathcal{O}_{6,0}^{+}\mathcal{O}_{6,0}P_{n_{1},n_{2},n_{3}}^{\left(\alpha,\beta,\gamma,\delta-1,a,b\right)  }(x,y,z)   &  =\left(n_{3}+\gamma \right)  \left(  n_{3}+\delta \right)  P_{n_{1},n_{2},n_{3}}^{\left(  \alpha,\beta,\gamma,\delta-1,a,b\right)  }(x,y,z),  \\
\mathcal{O}_{6,0}\mathcal{O}_{6,0}^{+}P_{n_{1},n_{2},n_{3}}^{\left(\alpha,\beta,\gamma-1,\delta,a,b\right)  }(x,y,z)   &  =\left(n_{3}+\gamma \right)  \left(  n_{3}+\delta \right)  P_{n_{1},n_{2},n_{3}}^{\left(  \alpha,\beta,\gamma-1,\delta,a,b\right)  }(x,y,z).
\end{align*}

\begin{theorem}
The polynomials $P_{n_{1},n_{2},n_{3}}^{(\alpha,\beta,\gamma,\delta,a,b)  }(x,y,z)  $ defined in \eqref{C} satisfy the following second order partial differential equations
\begin{multline}\label{AA1}
T_{1}u:= x(1-x)  u_{xx}+y(1-y)  u_{yy}+z(1-z)
u_{zz}-2xzu_{xz}-2yzu_{yz}-2xyu_{xy}\\
 +\left(  \alpha+1-\left(  \alpha+\beta+\gamma+\delta+a+b+4\right)  x\right)
u_{x}+\left(  \beta+1-\left(  \alpha+\beta+\gamma+\delta+4\right)
y+\frac{\left(  a+b\right)  xy-by}{1-x}\right)  u_{y}\\
 +\left(  \gamma+1-\left(  \alpha+\beta+\gamma+\delta+4\right)
z+\frac{\left(  a+b\right)  xz}{1-x}+\frac{byz}{(1-x)  \left(
1-x-y\right)  }\right)  u_{z}\\
 +\left(  n\left(  n+\alpha+\beta+\gamma+\delta+a+b+3\right)  -\frac
{a\left(  n_{2}+n_{3}\right)  }{1-x}-\frac{n_{3}b}{1-x-y}\right)  u=0,
\end{multline}

\begin{multline}\label{A2}
T_{2}u:=y\left(  1-x-y\right) u_{yy}-2yz u_{yz}+\dfrac{yz^{2}}{1-x-y}u_{zz}+\left(  \left(  \beta+1\right)  \left(  1-x-y\right)  -\left(  \gamma+\delta+b+2\right)  y\right)  u_{y}\\
+\dfrac{1}{1-x-y}\left(  \left(  \gamma+\delta+b+2\right)  yz-\left(\beta+1\right)  z\left(  1-x-y\right)  \right)  u_{z}\\
+\left(  \left(  \beta+1\right)  n_{3}+n_{2}\left(  n_{2}+2n_{3}+\beta+\gamma+\delta+b+2\right)  -\dfrac{n_{3}\left(  \gamma+\delta+b+n_{3}+1\right)  y}{1-x-y}\right)  u=0,
\end{multline}
\begin{equation}\label{A3}
T_{3}u:=z\left(  1-x-y-z\right)  u_{zz}+\left[  \left(  \gamma+1\right)  \left(1-x-y-z\right)  -\left(  \delta+1\right)  z\right]  u_{z}+n_{3}\left(n_{3}+\gamma+\delta+1\right)  u=0,
\end{equation}
and
\begin{multline}\label{A4}
T_{4}u:=x(1-x)  u_{xx}+\dfrac{xy^{2}}{1-x}u_{yy}+\dfrac{xz^{2}}{1-x}u_{zz}-2xy u_{xy}-2xz u_{xz}+\dfrac{2xyz}{1-x} u_{yz}\\
+\left(  \alpha+1-\left(  \alpha+\beta+\gamma+\delta+a+b+4\right)  x\right)u_{x}-\dfrac{y}{1-x}\left(  \alpha+1-\left(\alpha+\beta+\gamma+\delta+a+b+4\right)  x\right)  u_{y}\\
-\dfrac{z}{1-x}\left(  \alpha+1-\left(  \alpha+\beta+\gamma+\delta+a+b+4\right)  x\right)  u_{z}\\
+\left(  n\left(  n+\alpha+\beta+\gamma+\delta+a+b+3\right)  -\dfrac{\left(n_{2}+n_{3}\right)  \left(  n_{2}+n_{3}+\beta+\gamma+\delta+a+b+2\right)}{1-x}\right)  u=0.
\end{multline}
\end{theorem}

\begin{proof}
The first equation is obtained from
\[
-\mathcal{N}_{10}\mathcal{N}_{10}^{+}-\frac{1}{1-x}\mathcal{N}_{01}\mathcal{N}_{01}^{+}-\frac{1}{1-x-y}\mathcal{O}_{10}\mathcal{O}_{10}^{+}
\]
and, second and third equations follow from $\mathcal{N}_{01}\mathcal{N}_{01}^{+}$ and $\mathcal{O}_{10}\mathcal{O}_{10}^{+}$, respectively. From
\[
T_{1}u-\dfrac{1}{1-x}T_{2}u-\dfrac{1}{1-x-y}T_{3}u=0,
\]
we obtain the last equation.
\end{proof}

\begin{remark}
For $a=b=0$, the first equation reduces to the partial differential equation \cite{DX} for the polynomials $P_{n_{1},n_{2},n_{3}}^{\left(  \alpha,\beta,\gamma,\delta \right)  }(x,y,z)  $ defined in \eqref{C1}
\begin{multline*}
 x(1-x)  u_{xx}+y(1-y)  u_{yy}+z(1-z)u_{zz}-2xzu_{xz}-2yzu_{yz}-2xyu_{xy}\\
 +\left(  \alpha+1-\left(  \alpha+\beta+\gamma+\delta+4\right)  x\right)u_{x}+\left(  \beta+1-\left(  \alpha+\beta+\gamma+\delta+4\right)  y\right)u_{y}\\
 +\left(  \gamma+1-\left(  \alpha+\beta+\gamma+\delta+4\right)  z\right)u_{z}+n\left(  n+\alpha+\beta+\gamma+\delta+3\right)  u=0.
\end{multline*}
\end{remark}

\begin{remark}
For the second order differential equations in the first set given above, the first, second and the last two equations give the second order equation in the form of \eqref{A2}. The third, fourth, seventh and eighth equations give the equation \eqref{A2} multiplied by $y$. The fifth, sixth, ninth and tenth equations give the equation \eqref{A2} multiplied by $1-x-y$. For the second order differential equations in the second set, the first, second and the last two equations give the second order equation in the form of \eqref{A4}. The third, fourth, seventh and eighth equations give the equation \eqref{A4} multiplied by $x.$ The fifth, sixth, ninth and tenth equations give the equation \eqref{A4} multiplied by $1-x$. Moreover, for the second order differential equations given in the third set, the first, second and the last two equations give the second order equation in the form of \eqref{A3}. The third, fourth, seventh and eighth equations give the equation \eqref{A3} multiplied by $z$. The fifth, sixth, ninth and tenth equations give the equation \eqref{A3} multiplied by $1-x-y-z$.
\end{remark}

\begin{remark}
The partial differential equation \eqref{A4} has also a monic solution given by
\begin{multline}
\tilde{P}_{n_{1},n_{2},n_{3}}^{(\alpha,\beta,\gamma,\delta,a,b)}(x,y,z)=\frac{\left(  -1\right)  ^{n_{1}} \Gamma(\alpha+n_{1}+1)\Gamma(\alpha+\beta+\gamma+\delta+a+b+n+n_{2}+n_{3}+3)}
{\Gamma(\alpha+1)\Gamma(\alpha+\beta+\gamma+\delta+a+b+2n+3)}\\
\times y^{n_{2}}z^{n_{3}} \, \hyper{2}{1}{-n_{1},\alpha+\beta+\gamma+\delta+a+b+n+n_{2}+n_{3}+3}{\alpha+1}{x},
\end{multline}
which can be written as
\begin{multline}
\tilde{P}_{n_{1},n_{2},n_{3}}^{(\alpha,\beta,\gamma,\delta,a,b)}(x,y,z)=\frac{n_{1}! \Gamma(\alpha+\beta+\gamma+\delta+a+b+n+n_{2}+n_{3}+3)}
{\Gamma(\alpha+\beta+\gamma+\delta+a+b+2n+3)} \\
\times y^{n_{2}}z^{n_{3}}~\widetilde{P}_{n_{1}}^{\left(  \beta+\gamma+\delta+a+b+2n_{2}+2n_{3}+2,\alpha \right)  }(x)
\end{multline}
in terms of shifted Jacobi polynomials.
\end{remark}

\begin{theorem}
The following connection relation holds true
\begin{multline*}
P_{n_{1},n_{2},n_{3}}^{(\alpha,\beta,\gamma,\delta,a,b)  }(x,y,z)     =\sum \limits_{m=0}^{n_{1}}\frac{\left(  -1\right)  ^{m}\left(  2n-2m+e-\alpha+\xi+3\right)
\Gamma \left(  n+n_{2}+n_{3}+e-\alpha+3\right)  \left(  \alpha-\xi \right)_{m}}{  \Gamma \left(n+n_{2}+n_{3}+e+3\right)  \Gamma \left(  2n-m+e-\alpha+\xi+4\right)  }\\
  \times \frac{\Gamma \left(  2n-m+e+3\right)  \Gamma \left(  n+n_{2}+n_{3}-m+e-\alpha+\xi+3\right)}{m!\Gamma \left(  n+n_{2}+n_{3}-m+e-\alpha+3\right)}  P_{n_{1}-m,n_{2},n_{3}}^{\left(  \xi,\beta,\gamma,\delta,a,b\right)  }(x,y,z)
\end{multline*}
where $P_{n_{1},n_{2},n_{3}}^{(\alpha,\beta,\gamma,\delta,a,b)  }(x,y,z) $ are defined in \eqref{C}, $e=\alpha+\beta+\gamma+\delta+a+b$, $n=n_{1}+n_{2}+n_{3}$, and $\left(  \lambda \right)  _{n}$ denotes the Pochammer symbol defined by $\left(  \lambda \right)  _{n}=\lambda \left(  \lambda+1\right)  \cdots \left(\lambda+n-1\right)$; $n\in \mathbb{N}_{0}$, $\left(  \lambda \right)  _{0}=1$.
\end{theorem}

\begin{proof}
It follows from the connection relation between univariate Jacobi polynomials \cite{Askey,KS}
\begin{multline*}
P_{n}^{\left(  \alpha,\beta \right)  }(x)    ={\displaystyle \sum \limits_{m=0}^{n}}\frac{\left(  -1\right)  ^{n-m}\left(  2m+\alpha+\delta+1\right)
\Gamma \left(  n+\alpha+1\right)  \Gamma \left(  n+m+\alpha+\beta+1\right)}{\Gamma \left(  m+\alpha+1\right)  \Gamma \left(  n+\alpha+\beta+1\right)  }\\
  \times \frac{\Gamma \left(  m+\alpha+\delta+1\right)  \left(  \beta-\delta \right)  _{n-m}}{\Gamma \left(  n+m+\alpha+\delta+2\right)  \left(n-m\right)  !}P_{m}^{\left(  \alpha,\delta \right)  }(x).
\end{multline*}
\end{proof}

\begin{theorem}
The following connection relation holds true
\begin{multline*}
P_{n_{1},n_{2},n_{3}}^{(\alpha,\beta,\gamma,\delta,a,b)  }(x,y,z)   \\=
{\displaystyle \sum \limits_{k_{1}=0}^{n_{1}}}
{\displaystyle \sum \limits_{k_{2}=0}^{n_{2}}}
{\displaystyle \sum \limits_{k_{3}=0}^{n_{3}}}
\frac{\left(  k_{1}+2n_{2}+2n_{3}+\beta+\gamma+\delta+a+b+3\right)
_{n_{1}-k_{1}}\left(  n_{1}+2n_{2}+2n_{3}+e+3\right)  _{k_{1}}}{\left(  n_{1}-k_{1}\right)  !\left(  n_{2}-k_{2}\right)  !\left(  n_{3}-k_{3}\right)  !\left(  k_{1}+2k_{2}+2k_{3}
+\phi+\theta+\xi+\eta+a+b+3\right)  _{k_{1}}}\\
 \times \frac{\left(  k_{2}+2n_{3}+\gamma+\delta+b+2\right)  _{n_{2}-k_{2}}\left(  n_{2}+2n_{3}+\beta+\gamma+\delta+b+2\right)  _{k_{2}}\left(k_{3}+\delta+1\right)  _{n_{3}-k_{3}}\left(  n_{3}+\gamma+\delta+1\right)_{k_{3}}}{\left(  k_{2}+2k_{3}+\theta+\xi+\eta+b+2\right)  _{k_{2}}\left(k_{3}+\xi+\eta+1\right)  _{k_{3}}}\\
\times \hyper{3}{2}{k_{1}-n_{1} ,n_{1}+2n_{2}+2n_{3}+k_{1}+e+3,k_{1}+2k_{2}+2k_{3}+\theta+\xi+\eta+a+b+3}{2k_{1}+2k_{2}+2k_{3}+\phi+\theta+\xi+\eta+a+b+4,k_{1}+2n_{2}+2n_{3}+\beta+\gamma+\delta+a+b+3}{1}\\
\times \hyper{3}{2}{k_{2}-n_{2} ,n_{2}+2n_{3}+k_{2}+\beta+\gamma+\delta+b+2,k_{2}+2k_{3}+\xi+\eta+b+2}{2k_{2}+2k_{3}+\theta+\xi+\eta+b+3,k_{2}+2n_{3}+\gamma+\delta+b+2}{1}\\
\times \hyper{3}{2}{k_{3}-n_{3} ,n_{3}+k_{3}+\gamma+\delta+1,k_{3}+\xi+1}{2k_{3}+\xi+\eta+2,k_{3}+\delta+1}{1} \\
\times
 (1-x)  ^{n_{2}-k_{2}}\left(  1-x-y\right)  ^{n_{3}-k_{3}} P_{k_{1},k_{2},k_{3}}^{\left(  \phi,\theta,\eta,\xi,a,b\right)}(x,y,z),
\end{multline*}
where $P_{n_{1},n_{2},n_{3}}^{(\alpha,\beta,\gamma,\delta,a,b)  }(x,y,z) $ are defined in \eqref{C}.
\end{theorem}
\begin{proof}
The result is a consequence of the connection relation between univariate Jacobi polynomials (see \cite{L} and \cite{MEH})
\begin{multline*}
P_{n}^{\left(  \alpha,\beta \right)  }(x)  ={\displaystyle \sum \limits_{k=0}^{n}}\frac{\left(  k+\alpha+1\right)  _{n-k}\left(  n+\alpha+\beta+1\right)  _{k}
}{\left(  n-k\right)  !\left(  k+\gamma+\delta+1\right)  _{k}} \\ \times \hyper{3}{2}{k-n ,n+k+\alpha+\beta+1,k+\gamma+1}{2k+\gamma+\delta+2,k+\alpha+1}{1}P_{k}^{\left(
\gamma,\delta \right)  }(x).
\end{multline*}
\end{proof}

\begin{theorem}
For $n\geq0,$ the three term recurrence relation holds
\begin{multline*}
xP_{n_{1},n_{2},n_{3}}^{(\alpha,\beta,\gamma,\delta,a,b)}(x,y,z)  =\frac{(n_{1}+1)(e+n+n_{2}+n_{3}+3)}{\left(  e+2n+3\right)  \left(
e+2n+4\right)  }P_{n_{1}+1,n_{2},n_{3}}^{(\alpha,\beta,\gamma,\delta,a,b)  }(x,y,z) \\
+\frac{(\alpha+2n_{1}+1)(e+2n+2)-2n_{1}(\alpha+n_{1})}{\left(  e+2n+2\right)  \left(e+2n+4\right)  }
P_{n_{1},n_{2},n_{3}}^{(\alpha,\beta,\gamma,\delta,a,b)  }\left(x,y,z\right) \\
+\frac{\left(  n+n_{2}+n_{3}+e-\alpha+2\right)  \left(  \alpha+n_{1}\right)  }{\left(  e+2n+2\right)  \left(  e+2n+3\right)  }P_{n_{1}-1,n_{2},n_{3}}^{(\alpha,\beta,\gamma,\delta,a,b)}(x,y,z),
\end{multline*}
where $n=n_{1}+n_{2}+n_{3}$ and $e=\alpha+\beta+\gamma+\delta+a+b$.
\end{theorem}
\begin{proof}
It is enough to use the three term recurrence relation \cite[p.263, Eq. (1)]{R} for the univariate Jacobi polynomials.
\end{proof}

\section{Sparse recurrence relations for $P_{n_{1},n_{2},n_{3}}^{\left(\alpha,\beta,\gamma,\delta \right)}(x,y,z)  $}\label{sec:4}

By combining the results given in the previous section, it is possible to derive sparse recurrence relations between Koornwinder polynomials in three variables with different parameters and their partial derivatives. We denote $w=1-x-y-z$ in the following results.

\begin{corollary}
The partial derivatives of $P_{n_{1},n_{2},n_{3}}^{\left(  \alpha,\beta,\gamma,\delta \right)  }(x,y,z)  $ can be expressed in terms of Koornwinder polynomials in three variables with incremented parameters as follows
\begin{multline} \label{21}
 \left(  2n_{2}+2n_{3}+\beta+\gamma+\delta+2\right)  \left(  \frac{\partial}{\partial x}-\frac{\partial}{\partial y}\right)  P_{n_{1},n_{2},n_{3}}^{\left(  \alpha,\beta,\gamma,\delta \right)  }(x,y,z) \\
  =\left(  n_{2}+2n_{3}+\beta+\gamma+\delta+2\right)  \left(  n+n_{2}+n_{3}+\alpha+\beta+\gamma+\delta+3\right)  P_{n_{1}-1,n_{2},n_{3}}^{\left(\alpha+1,\beta+1,\gamma,\delta \right)  }(x,y,z)  \\
-\left(  n_{1}+2n_{2}+2n_{3}+\beta+\gamma+\delta+2\right)  \left(n_{2}+2n_{3}+\gamma+\delta+1\right)  P_{n_{1},n_{2}-1,n_{3}}^{\left(\alpha+1,\beta+1,\gamma,\delta \right)  }(x,y,z)  ,
\end{multline}
\begin{multline}\label{3}
\left(  2n_{3}+\gamma+\delta+1\right)  \left(  \frac{\partial}{\partial z}-\frac{\partial}{\partial y}\right)  P_{n_{1},n_{2},n_{3}}^{\left(\alpha,\beta,\gamma,\delta \right)  }(x,y,z)  \\
=\left(  n_{3}+\delta \right)  \left(  n_{2}+2n_{3}+\gamma+\delta+1\right)P_{n_{1},n_{2},n_{3}-1}^{\left(  \alpha,\beta+1,\gamma+1,\delta \right)}(x,y,z)  \\
-\left(  n_{2}+2n_{3}+\beta+\gamma+\delta+2\right)  \left(  n_{3}+\gamma+\delta+1\right)  P_{n_{1},n_{2}-1,n_{3}}^{\left(  \alpha,\beta+1,\gamma+1,\delta \right)  }(x,y,z)  ,
\end{multline}
\begin{equation}\label{4}
\frac{\partial}{\partial z}P_{n_{1},n_{2},n_{3}}^{\left(  \alpha,\beta ,\gamma,\delta \right)  }(x,y,z)  =\left(  n_{3}+\gamma +\delta+1\right)  P_{n_{1},n_{2},n_{3}-1}^{\left(  \alpha,\beta,\gamma +1,\delta+1\right)  }(x,y,z)  ,
\end{equation}
and
\begin{multline}\label{2}
\left(  2n_{2}+2n_{3}+\beta+\gamma+\delta+2\right) \frac{\partial}{\partial z} \left(  \frac{\partial}{\partial x}-\frac{\partial}{\partial y}\right)
P_{n_{1},n_{2},n_{3}}^{\left(  \alpha,\beta,\gamma,\delta \right)  }(x,y,z) \\
  =\left(  n_{2}+2n_{3}+\beta+\gamma+\delta+2\right)  \left(  n+n_{2}+n_{3}+\alpha+\beta+\gamma+\delta+3\right) \\
  \times \left(  n_{3}+\gamma+\delta+1\right)  P_{n_{1}-1,n_{2},n_{3}-1}^{\left(\alpha+1,\beta+1,\gamma+1,\delta+1 \right)  }(x,y,z)  \\
 -\left(  n_{1}+2n_{2}+2n_{3}+\beta+\gamma+\delta+2\right)  \left(n_{2}+2n_{3}+\gamma+\delta+1\right) \\
 \times \left(  n_{3}+\gamma+\delta+1\right)  P_{n_{1},n_{2}-1,n_{3}-1}^{\left(\alpha+1,\beta+1,\gamma+1,\delta+1 \right)  }(x,y,z) .
\end{multline}
\end{corollary}
\begin{proof}
The first relation comes from the equality
\[
\left(  \mathcal{N}_{1,0}\mathcal{N}_{0,3}-\mathcal{N}_{0,5}^{+}\mathcal{N}_{6,0}^{+}\right)  u=\left(  2n_{2}+2n_{3}+\beta+\gamma+\delta+2\right)  \left(  \dfrac{\partial}{\partial x}-\dfrac{\partial}{\partial y}\right)  u
\]
when $a=b=0$. The relation \eqref{3} is obtained when $a=b=0$ since
\[
\left(  \mathcal{O}_{5,0}^{+}\mathcal{N}_{0,6}^{+}-\mathcal{N}_{0,1}\mathcal{O}_{3,0}\right)  u=\left(  2n_{3}+\gamma+\delta+1\right)  \left(\dfrac{\partial}{\partial z}-\dfrac{\partial}{\partial y}\right)  u.
\]
The relation given by $\mathcal{O}_{1,0}$ in Theorem \ref{Theorem 1} for $a=b=0$ gives (\ref{4}). Finally, for the last relation, it is enough to combine the first and the third relations.
\end{proof}

The derivatives of weighted versions of $P_{n_{1},n_{2},n_{3}}^{\left(\alpha,\beta,\gamma,\delta \right)  }(x,y,z)  $ verify the following sparse recurrence relations.
\begin{corollary}
The following relations hold true
\begin{multline}
\left(  2n_{2}+2n_{3}+\beta+\gamma+\delta+2\right)  \left(  \frac{\partial}{\partial x}-\frac{\partial}{\partial y}\right)  \left(  x^{\alpha}y^{\beta}z^{\gamma}w^{\delta}P_{n_{1},n_{2},n_{3}}^{\left(  \alpha,\beta,\gamma,\delta \right)  }(x,y,z)  \right) \\
=x^{\alpha-1}y^{\beta-1}z^{\gamma}w^{\delta}\left \{  \left(  n_{1}+\alpha \right)  \left(  n_{2}+1\right)  P_{n_{1},n_{2}+1,n_{3}}^{\left(\alpha-1,\beta-1,\gamma,\delta \right)  }(x,y,z)  \right. \\
\left.  -\left(  n_{1}+1\right)  \left(  n_{2}+\beta \right)  P_{n_{1}+1,n_{2},n_{3}}^{\left(  \alpha-1,\beta-1,\gamma,\delta \right)  }\left(x,y,z\right)  \right \}  ,
\end{multline}
\begin{multline}
\left(  2n_{3}+\gamma+\delta+1\right)  \left(  \frac{\partial}{\partial z}-\frac{\partial}{\partial y}\right)  \left(  x^{\alpha}y^{\beta}z^{\gamma}w^{\delta}P_{n_{1},n_{2},n_{3}}^{\left(  \alpha,\beta,\gamma,\delta \right)}(x,y,z)  \right) \\
=x^{\alpha}y^{\beta-1}z^{\gamma-1}w^{\delta}\left \{ - \left(  n_{2}+\beta \right)  \left(  n_{3}+1\right)  P_{n_{1},n_{2},n_{3}+1}^{\left(
\alpha,\beta-1,\gamma-1,\delta \right)  }(x,y,z)  \right. \\
\left.  +\left(  n_{2}+1\right)  \left(  n_{3}+\gamma \right) P_{n_{1},n_{2}+1,n_{3}}^{\left(  \alpha,\beta-1,\gamma-1,\delta \right)}(x,y,z)  \right \},
\end{multline}
\begin{equation}
\frac{\partial}{\partial z}\left(  x^{\alpha}y^{\beta}z^{\gamma}w^{\delta}P_{n_{1},n_{2},n_{3}}^{\left(  \alpha,\beta,\gamma,\delta \right)  }\left(
x,y,z\right)  \right)  =-x^{\alpha}y^{\beta}z^{\gamma-1}w^{\delta-1}\left(n_{3}+1\right)  P_{n_{1},n_{2},n_{3}+1}^{\left(  \alpha,\beta,\gamma
-1,\delta-1\right)  }(x,y,z),
\end{equation}
and
\begin{multline}
\left(  2n_{2}+2n_{3}+\beta+\gamma+\delta+2\right)  \frac{\partial}{\partial z}\left(  \frac{\partial}{\partial x}-\frac{\partial}{\partial y}\right)  \left(  x^{\alpha}y^{\beta
}z^{\gamma}w^{\delta}P_{n_{1},n_{2},n_{3}}^{\left(  \alpha,\beta,\gamma,\delta \right)  }(x,y,z)  \right) \\
=-x^{\alpha-1}y^{\beta-1}z^{\gamma-1}w^{\delta-1}\left \{  \left(  n_{1}+\alpha \right)  \left(  n_{2}+1\right)\left(  n_{3}+1\right)  P_{n_{1},n_{2}+1,n_{3}+1}^{\left(
\alpha-1,\beta-1,\gamma-1,\delta-1 \right)  }(x,y,z)  \right. \\
\left.  -\left(  n_{1}+1\right)  \left(  n_{2}+\beta \right)\left(  n_{3}+1\right)  P_{n_{1}+1,n_{2},n_{3}+1}^{\left(  \alpha-1,\beta-1,\gamma-1,\delta-1 \right)  }\left(
x,y,z\right)  \right \}.
\end{multline}
\end{corollary}
\begin{proof}
The first relation is a consequence of
\begin{multline*}
\left(  \mathcal{N}_{1,0}^{+}\mathcal{N}_{0,3}^{+}-\mathcal{N}_{0,5}\mathcal{N}_{6,0}\right)  u  =\left(  2n_{2}+2n_{3}+\beta+\gamma
+\delta+2\right)  \left(  xy\left(  \frac{\partial}{\partial y}-\frac{\partial}{\partial x}\right)  +\beta x-\alpha y\right)  u\\
=-\left(  2n_{2}+2n_{3}+\beta+\gamma+\delta+2\right)  x^{1-\alpha}y^{1-\beta}w^{-\delta}\left(  \frac{\partial}{\partial x}-\frac{\partial}{\partial y}\right)  \left(  x^{\alpha}y^{\beta}w^{\delta}u\right)  .
\end{multline*}
The second one comes from the relation
\begin{multline*}
\left(  \mathcal{O}_{5,0}\mathcal{N}_{0,6}-\mathcal{N}_{0,1}^{+}\mathcal{O}_{3,0}^{+}\right)  u  =\left(  2n_{3}+\gamma+\delta+1\right)  \left(  yz\left(  \frac{\partial}{\partial y}-\frac{\partial}{\partial z}\right)  +\beta z-\gamma y\right)  u\\
=-\left(  2n_{3}+\gamma+\delta+1\right)  y^{1-\beta}z^{1-\gamma}w^{-\delta}\left(  \frac{\partial}{\partial z}-\frac{\partial}{\partial y}\right)\left(  y^{\beta}z^{\gamma}w^{\delta}u\right)  .
\end{multline*}
The third relation holds from the relation
\[
\mathcal{O}_{1,0}^{+}u=\left(  z\delta-w\gamma \right)  u-zw\frac{\partial u}{\partial z}=-z^{1-\gamma}w^{1-\delta}\frac{\partial}{\partial z}\left( z^{\gamma}w^{\delta}u\right).
\]
By taking into account the first and the third relations, we obtain the last relation.
\end{proof}

It is possible to write $xP_{n_{1},n_{2},n_{3}}^{\left(  \alpha,\beta,\gamma,\delta \right)  }(x,y,z)$, $yP_{n_{1},n_{2},n_{3}}^{\left(  \alpha,\beta,\gamma,\delta \right)  }(x,y,z)$, $zP_{n_{1},n_{2},n_{3}}^{\left(  \alpha,\beta,\gamma,\delta \right)  }\left(x,y,z\right) $, and $wP_{n_{1},n_{2},n_{3}}^{\left(  \alpha,\beta,\gamma,\delta \right)  }\left(x,y,z\right)  $  in terms of Koornwinder polynomials with different parameters as follows.

\begin{corollary}
The following relations are satisfied
\begin{multline}
\left(  2n+\alpha+\beta+\gamma+\delta+3\right)  xP_{n_{1},n_{2},n_{3}}^{\left(  \alpha,\beta,\gamma,\delta \right)  }(x,y,z)
=\left(  n_{1}+\alpha \right)  P_{n_{1},n_{2},n_{3}}^{\left(  \alpha-1,\beta,\gamma,\delta \right)  }(x,y,z)  \\
 +\left(  n_{1}+1\right)  P_{n_{1}+1,n_{2},n_{3}}^{\left(  \alpha-1,\beta,\gamma,\delta \right)  }(x,y,z)  ,
\end{multline}
\begin{multline}
\left(  2n+\alpha+\beta+\gamma+\delta+3\right)  \left(  2n_{2}+2n_{3}+\beta+\gamma+\delta+2\right)  yP_{n_{1},n_{2},n_{3}}^{\left(  \alpha,\beta,\gamma,\delta \right)  }(x,y,z)  \\
=\left(  n+n_{2}+n_{3}+\beta+\gamma+\delta+2\right)  \left(  n_{2}+\beta \right)  P_{n_{1},n_{2},n_{3}}^{\left(  \alpha,\beta-1,\gamma,\delta \right)  }(x,y,z)  \\
+\left(  n+n_{2}+n_{3}+\alpha+\beta+\gamma+\delta+3\right)  \left(n_{2}+1\right)  P_{n_{1},n_{2}+1,n_{3}}^{\left(  \alpha,\beta-1,\gamma,\delta \right)  }(x,y,z)  \\
-\left(  n_{1}+1\right)  \left(  n_{2}+\beta \right)  P_{n_{1}+1,n_{2},n_{3}}^{\left(  \alpha,\beta-1,\gamma,\delta \right)  }(x,y,z)
-\left(  n_{1}+\alpha \right)  \left(  n_{2}+1\right)  P_{n_{1}-1,n_{2}+1,n_{3}}^{\left(  \alpha,\beta-1,\gamma,\delta \right)  }(x,y,z),
\end{multline}
\begin{multline}
\left(  2n+\alpha+\beta+\gamma+\delta+3\right)  \left(  2n_{2}+2n_{3}+\beta+\gamma+\delta+2\right)  \left(  \gamma+\delta+2n_{3}+1\right)  zP_{n_{1},n_{2},n_{3}}^{\left(  \alpha,\beta,\gamma,\delta \right)  }\left(x,y,z\right)  \\
=\left(  n+n_{2}+n_{3}+\beta+\gamma+\delta+2\right)  \left(  n_{2}+2n_{3}+\gamma+\delta+1\right)  \left(  n_{3}+\gamma \right)  P_{n_{1},n_{2},n_{3}}^{\left(  \alpha,\beta,\gamma-1,\delta \right)  }\left(x,y,z\right)  \\
-\left(  n_{1}+1\right)  \left(  n_{2}+2n_{3}+\gamma+\delta+1\right)\left(  n_{3}+\gamma \right)  P_{n_{1}+1,n_{2},n_{3}}^{\left(  \alpha,\beta,\gamma-1,\delta \right)  }(x,y,z)  \\
-\left(  n+n_{2}+n_{3}+\alpha+\beta+\gamma+\delta+3\right)  \left(n_{2}+1\right)  \left(  n_{3}+\gamma \right)  P_{n_{1},n_{2}+1,n_{3}}^{\left(
\alpha,\beta,\gamma-1,\delta \right)  }(x,y,z)  \\
+\left(  n_{1}+\alpha \right)  \left(  n_{2}+1\right)  \left(  n_{3}+\gamma \right)  P_{n_{1}-1,n_{2}+1,n_{3}}^{\left(  \alpha,\beta,\gamma-1,\delta \right)  }(x,y,z)  \\
+\left(  n+n_{2}+n_{3}+\alpha+\beta+\gamma+\delta+3\right)  \left(n_{2}+2n_{3}+\beta+\gamma+\delta+2\right)  \left(  n_{3}+1\right)P_{n_{1},n_{2},n_{3}+1}^{\left(  \alpha,\beta,\gamma-1,\delta \right)  }\left(x,y,z\right)  \\
-\left(  n_{1}+\alpha \right)  \left(  n_{2}+2n_{3}+\beta+\gamma+\delta+2\right)  \left(  n_{3}+1\right)  P_{n_{1}-1,n_{2},n_{3}+1}^{\left(\alpha,\beta,\gamma-1,\delta \right)  }(x,y,z)  \\
-\left(  n+n_{2}+n_{3}+\beta+\gamma+\delta+2\right)  \left(  n_{2}+\beta \right)  \left(  n_{3}+1\right)  P_{n_{1},n_{2}-1,n_{3}+1}^{\left(\alpha,\beta,\gamma-1,\delta \right)  }(x,y,z)  \\
+\left(  n_{1}+1\right)  \left(  n_{2}+\beta \right)  \left(  n_{3}+1\right)  P_{n_{1}+1,n_{2}-1,n_{3}+1}^{\left(  \alpha,\beta,\gamma-1,\delta \right)  }(x,y,z),
\end{multline}
and
\begin{multline}
\left(  2n+\alpha+\beta+\gamma+\delta+3\right)  \left(  2n_{2}+2n_{3}+\beta+\gamma+\delta+2\right)  \left(  \gamma+\delta+2n_{3}+1\right)wP_{n_{1},n_{2},n_{3}}^{\left(  \alpha,\beta,\gamma,\delta \right)  }\left(x,y,z\right)  \\
=-\left(  n+n_{2}+n_{3}+\alpha+\beta+\gamma+\delta+3\right)  \left(n_{2}+2n_{3}+\beta+\gamma+\delta+2\right)  \left(  n_{3}+1\right)
P_{n_{1},n_{2},n_{3}+1}^{\left(  \alpha,\beta,\gamma,\delta-1\right)  }\left(x,y,z\right)  \\
+\left(  n_{1}+\alpha \right)  \left(  n_{2}+2n_{3}+\beta+\gamma+\delta+2\right)  \left(  n_{3}+1\right)  P_{n_{1}-1,n_{2},n_{3}+1}^{\left(\alpha,\beta,\gamma,\delta-1\right)  }(x,y,z)  \\
+\left(  n+n_{2}+n_{3}+\beta+\gamma+\delta+2\right)  \left(  n_{2}+\beta \right)  \left(  n_{3}+1\right)  P_{n_{1},n_{2}-1,n_{3}+1}^{\left(\alpha,\beta,\gamma,\delta-1\right)  }(x,y,z)  \\
-\left(  n_{1}+1\right)  \left(  n_{2}+\beta \right)  \left(  n_{3}+1\right)  P_{n_{1}+1,n_{2}-1,n_{3}+1}^{\left(  \alpha,\beta,\gamma,\delta-1\right)  }(x,y,z)  \\
+\left(  n+n_{2}+n_{3}+\beta+\gamma+\delta+2\right)  \left(  n_{2}+2n_{3}+\gamma+\delta+1\right)  \left(  n_{3}+\delta \right)  P_{n_{1},n_{2},n_{3}}^{\left(  \alpha,\beta,\gamma,\delta-1\right)  }\left(x,y,z\right)  \\
-\left(  n_{1}+1\right)  \left(  n_{2}+2n_{3}+\gamma+\delta+1\right)\left(  n_{3}+\delta \right)  P_{n_{1}+1,n_{2},n_{3}}^{\left(  \alpha,\beta,\gamma,\delta-1\right)  }(x,y,z)  \\
-\left(  n+n_{2}+n_{3}+\alpha+\beta+\gamma+\delta+3\right)  \left(n_{2}+1\right)  \left(  n_{3}+\delta \right)  P_{n_{1},n_{2}+1,n_{3}}^{\left(\alpha,\beta,\gamma,\delta-1\right)  }(x,y,z)  \\
+\left(  n_{1}+\alpha \right)  \left(  n_{2}+1\right)  \left(  n_{3}+\delta \right)  P_{n_{1}-1,n_{2}+1,n_{3}}^{\left(  \alpha,\beta,\gamma,\delta-1\right)  }(x,y,z)  .
\end{multline}
\end{corollary}
\begin{proof}
The first relation follows from $\left(  \mathcal{N}_{3,0}^{+}+\mathcal{N}_{5,0}\right)  u=\left(  2n+\alpha+\beta+\gamma+\delta+3\right)  xu$. The second relation is satisfied by the equality
\begin{multline*}
\left(  2n+\alpha+\beta+\gamma+\delta+3\right)  \left(  2n_{2}+2n_{3}+\beta+\gamma+\delta+2\right)  yu\\
=\mathcal{N}_{0,3}^{+}\left(  \mathcal{N}_{2,0}^{+}-\mathcal{N}_{4,0}\right)  u+\mathcal{N}_{0,5}\left(  \mathcal{N}_{2,0}-\mathcal{N}_{4,0}^{+}\right)  u.
\end{multline*}
The third relation comes from
\begin{multline*}
\left(  2n+\alpha+\beta+\gamma+\delta+3\right)  \left(  2n_{2}+2n_{3}+\beta+\gamma+\delta+2\right)  \left(  \gamma+\delta+2n_{3}+1\right)  zu\\
=\left(  \mathcal{O}_{3,0}^{+}\mathcal{N}_{0,2}^{+}-\mathcal{O}_{5,0}\mathcal{N}_{0,4}^{+}\right)  \left(  \mathcal{N}_{2,0}^{+}-\mathcal{N}_{4,0}\right)  u-\left(  \mathcal{O}_{3,0}^{+}\mathcal{N}_{0,4}-\mathcal{O}_{5,0}\mathcal{N}_{0,2}\right)  \left(  \mathcal{N}_{2,0}-\mathcal{N}_{4,0}^{+}\right)  u.
\end{multline*}
The last one is obtained from the relation
\begin{multline*}
\left(  2n+\alpha+\beta+\gamma+\delta+3\right)  \left(  2n_{2}+2n_{3}+\beta+\gamma+\delta+2\right)  \left(  \gamma+\delta+2n_{3}+1\right)  wu\\
=\left(  \mathcal{O}_{4,0}\mathcal{N}_{0,4}^{+}+\mathcal{O}_{2,0}^{+}\mathcal{N}_{0,2}^{+}\right)  \left(  \mathcal{N}_{2,0}^{+}-\mathcal{N}_{4,0}\right)  u-\left(  \mathcal{O}_{4,0}\mathcal{N}_{0,2}+\mathcal{O}_{2,0}^{+}\mathcal{N}_{0,4}\right)  \left(  \mathcal{N}_{2,0}-\mathcal{N}_{4,0}^{+}\right)  u.
\end{multline*}

\end{proof}

\section*{Acknowledgements}

The work of the first author has been partially supported by The Scientific and Technological Research Council of Turkey. The work of the second author has been partially supported by the Agencia Estatal de Innovaci\'on (AEI) of Spain under Grant MTM2016-75140-P, cofinanced by the European Community fund FEDER.

\end{document}